\documentclass[a4paper, 11pt]{article}
\usepackage{mathrsfs,amssymb,amsmath, amsthm,color,tikz}

\usepackage[top=1.31in,bottom=1.29in,left=1.1in,right=0.9in]{geometry}
\usepackage[numbers]{natbib}
\setcitestyle{open={},close={}}

\numberwithin{equation}{section}\theoremstyle{definition}
\swapnumbers

 \newtheorem{Theorem}[equation]{Theorem}
 \newtheorem{Prop}[equation]{Proposition}
 \newtheorem{Lemma}[equation]{Lemma}
 \newtheorem{Cor}[equation]{Corollary}
 \newtheorem{Notation}[equation]{Notation}
 \newtheorem{Defn}[equation]{Definition}
 \newtheorem{Example}[equation]{Example}
 \newtheorem{Remark}[equation]{Remark}

 \newtheorem{Hypo}[equation]{Hypothesis}
\newtheorem{Results}[equation]{Results}
\newtheorem{Cons}[equation]{Consequences}

 \makeatletter
\def\enumerate{\begingroup\ifnum\@enumdepth>3\@toodeep\else
      \advance\@enumdepth\@ne
      \edef\@enumctr{enum\romannumeral\the\@enumdepth}%
      \topsep\z@\parskip\z@
      \list{\csname label\@enumctr\endcsname}
        {\@nmbrlisttrue\let\@listctr\@enumctr
         \parsep\z@\itemsep\z@\topsep\z@
         \setcounter{\@enumctr}{0}
         \def \fMakelabel##1{\hss\llap{\rm ##1}}
       }\fi}

 \makeatother

\def\t{\mathfrak t}

\def\s{\mathfrak s}
\def\p{\mathfrak p}

\def\C{\mathbb C}
\def\N{\mathbb N}

\def\Z{\mathbb Z}
\def\F{\mathbb F}

\def\cO{\mathcal O}

\def\cR{\mathcal R}
\def\cE{\mathcal E}

\def\cD{\mathcal D}
\def\cF{\mathcal F}

\def\fX{\mathfrak X}

\def\fM{\mathfrak M}
\def\fD{\mathfrak D}
\def\fS{\mathfrak S}

\def\fD{\mathfrak D}

\def\ff{\mathfrak f}

\def\d{\text{d}}

\def \ppi{\p_{_{ \mathcal I}}}
\def \pj{\p_{\!_{ \mathcal J}}}

\def\ss{\underline{\s}}
\def\tt{\underline{\t}}

\DeclareMathOperator{\Hom}{Hom}
\DeclareMathOperator{\tab}{tab}
\DeclareMathOperator{\Res}{Res}

\DeclareMathOperator{\Ind}{Ind}

\DeclareMathOperator{\row}{row}

\DeclareMathOperator{\RStd}{RStd}
\DeclareMathOperator{\End}{End}

\DeclareMathOperator{\Stab}{Stab}
\DeclareMathOperator{\Pstab}{Pstab}

\DeclareMathOperator{\main}{main}

\DeclareMathOperator{\Tabl}{Tabl}
\DeclareMathOperator{\supp}{supp}
\DeclareMathOperator{\im}{im}
\DeclareMathOperator{\Lie}{Lie}
\DeclareMathOperator{\id}{id}

\DeclareMathOperator{\comp}{comp}

\begin{document}
\setlength{\abovedisplayskip}{4pt}
\setlength{\belowdisplayskip}{4pt}
\parindent=0pt

\title{$U_n(q)$ acting on flags and supercharacters}

\author{Richard Dipper $^{a,*}$, Qiong Guo $^{b,*}$\\ 
\\$^{a}$ {\footnotesize Institut f\"{u}r Algebra und Zahlentheorie}\\ {\footnotesize Universit\"{a}t Stuttgart, 70569 Stuttgart, Germany}
\\ \scriptsize{E-mail: richard.dipper@mathematik.uni-stuttgart.de}
\\$^{b}$ {\footnotesize College of Sciences, Shanghai Institute of Technology} \\ {\footnotesize 201418 Shanghai, P. R. China}
\\ \scriptsize{E-mail:qiongguo@sit.edu.cn}
\setcounter{footnote}{-1}\footnote{
$^*${\scriptsize Corresponding author.}}
\setcounter{footnote}{-1}\footnote{
{\scriptsize This work was partially supported by the 
DFG priority programme SPP 1388 in representation theory, no. 99028426}}
\setcounter{footnote}{-1}
\setcounter{footnote}{-1}}
\date{}


 \maketitle

\begin{center} Dedicated to the memory of J.A.Green\end{center}

\begin{abstract}
Let $U=U_n(q)$ be the group of lower unitriangular $n\times n$-matrices with entries in the field $\F_q$ with $q$ elements for some prime power $q$ and $n\in \N$. We investigate the restriction to $U$ of the permutation action of $GL_n(q)$ on flags in the natural $GL_n(q)$-module $\F_q^n$. Applying our results to the special case of flags of length two we obtain a complete decomposition of the permutation representation of $GL_n(q)$ on the cosets of maximal parabolic subgroups into irreducible $\C U$-modules.
\end{abstract}

\section{Introduction}

The irreducible complex characters of the finite general linear group $GL_n(q)$ have been determined by J.A.~Green in his landmark paper [\cite{Gr55}].
They subdivide naturally into families, called Harish-Chandra series, labelled by conjugacy classes of semisimple elements of $GL_n(q)$.
In this paper we are  mainly concerned with the series attached to the identity element of $GL_n(q)$.
Those are afforded precisely by the irreducible constituents of the permutation representation of $GL_n(q)$ acting on the cosets of a Borel subgroup $B$, which are called unipotent Specht modules.
Many aspects of the representation theory of $GL_n(q)$ may be described analogously to the representation theory of the symmetric group $\fS_n$ -- indeed one expects the theory of $GL_n(q)$ to translate into that of $\fS_n$ by setting $q$ equal to $1$.

It is a classical result that the Specht module  $S^\lambda$ for $\fS_n$, where $\lambda$ is a partition of $n$, is integrally defined. It comes with a certain distinguished integral basis, called standard basis, which is labelled by standard $\lambda$-tableaux.
This suggests that a $q$-version of this should hold for unipotent Specht modules $S(\lambda)$ for $GL_n(q)$.
More precisely,  to each standard $\lambda$-tableau $\s$, there should be attached a polynomial $g_\s(x)\in \Z[x]$ with $g_\s(1)=1$ and  $g_\s(q)$ many elements of $S(\lambda)$ such that all these elements form a basis of $S(\lambda)$, where $\s$ runs through all the standard $\lambda$-tableaux.
In [\cite{brandt}, \cite{brandt2}, \cite{dj1}, \cite{guo}] such standard bases for $S(\lambda)$ were constructed in the special case of 2-part partitions $\lambda$.

One way to define unipotent Specht modules for $GL_n(q)$ in  a characteristic free way is provided by James' kernel intersection theorem [\cite{jamesbook}, Theorem 15.19]: For a partition $\lambda$ of $n$, the unipotent Specht module $S(\lambda)$ for $GL_n(q)$ over any field $K$ of characteristic not dividing $q$ is defined as intersection of the kernels of all homomorphisms $\phi: M(\lambda)\rightarrow M(\mu)$ for all partition $\mu$ of $n$ dominating $\lambda$. Here $M(\lambda)$ denotes the permutation representation of $GL_n(q)$ acting by right translation on the set of all $\lambda$-flags in the natural $GL_n(q)$-module $V=\F_q^n$, where $\F_q$ is the field with $q$ elements.  In a recent paper, Andrews [\cite{scott}] gave an alternate construction of the unipotent Specht modules based on generalized Gelfand-Graev characters.
If $K=\C$ then $S(\lambda)$ are irreducible unipotent $GL_n(q)$-modules appearing as irreducible constituents of  the permutation module of $GL_n(q)$ acting on flags in $V$.

An important ingredient  in [\cite{guo}] consists of an analysis of the restriction of $M(\lambda)$ to the unitriangular group $U=U_n(q)$ of $GL_n(q)$. The group algebra $KU$ is semisimple  for all fields $K$ of characteristic not dividing $q$. In  [\cite{guo}] decomposing the restriction of $M(\lambda)$ to $U$ completely into irreducible $U$-modules for 2-part partitions $\lambda$, and then applying James' kernel intersection theorem, the main results of [\cite{brandt2}] were reproved, giving the standard basis of $S(\lambda)$ a representation theoretic interpretation. This new approach in [\cite{guo}] bears resemblance to Kirillov's orbit method [\cite{kirillov}] and the supercharacter theory of $U$ introduced by Andr\'{e} [\cite{andre1}] and Yan [\cite{yan}]. Thus making this remarkable connection precise and generalize it to arbitrary compositions $\lambda$ of $n$ is hoped to be a first step towards constructing standard bases of unipotent Specht modules. This is the main goal of this paper.

The permutation modules $M(\lambda)$ decompose as $U$-modules into direct sums of $\s$-components $M_\s$, where $\s$ runs through the set $\RStd(\lambda)$ of row standard $\lambda$-tableaux. If, for $\s\in \RStd(\lambda)$, $d=d(\s)\in \fS_n $ denotes the permutation taking the initial $\lambda$-tableau to $\s$, the intersection $U^d\cap U$ is a semidirect product 
$U_K = U_J\rtimes U_L$ of pattern subgroups. In section 2 and 3 we monomialize the trivial representation of $U_L$ induced to $U_K$ in a nontrivial way applying and generalizing a result of Jedlitschky in [\cite{markus}]. In section 5 we describe the $U_K$-action on the monomial basis of $M_\s$ combinatorially in terms of $\lambda$-flags in $V$. For the special case of $\lambda = (1^n)$ and the initial $\lambda$-tableau we recover the Andr\`{e}-Yan supercharacters  of $U$ as characters afforded by the orbit modules (section 6).

In the remaining sections 7 and 8 we return to 2-part partitions improving on some of the results in [\cite{guo}]. Indeed we show that in this case the $U$-modules $M_\s$, $\s\in \RStd(\lambda)$, are multiplicity free and we determine their irreducible $U$-constituents in terms of supercharacters. Moreover we calculate combinatorially the multiplicities of irreducible $U$-constituents of $M(\lambda)$ and prove that these are independent of $q$.

\section{Monomial linearisation}\label{sec1}
In this section we shall present a procedure which allows to transform under certain circumstances a transitive $G$-set, $G$ a finite group, into a monomial $G$-set which may decompose into many orbits. Thus this yields a decomposition of the corresponding permutation module into a direct sum of monomial $\C G$-modules. This procedure is based on work of Jedlitschky [\cite{markus}]	 and Kirillov's orbit method [\cite{kirillov}]. For the moment, let $G$ be an arbitrary finite group and let $K$ be a field. Identifying the group algebra $KG$ with the $K$-algebra $K^G$ of functions from $G$ to $K$ by


\begin{equation}\label{1.1}
\tau\mapsto \sum_{g\in G}\tau(g) g \quad \text{ for }\tau\in K^G, \end{equation}
the following lemma is seen immediately:

\begin{Lemma}\label{3.2}
Let $G$ be a finite group, $H\leqslant G$ and let 
$e=\sum_{h\in H} h$ and $M=eK^G$. Then $M$ consists of functions from $G$ to $K$, which are constant on the right cosets of $H$ in $G$.\end{Lemma}

Now let $G$ be a finite group acting on a finite abelian group $(V,+)$ by automorphisms, the action of $g\in G$ denoted by $v\mapsto vg \,(v\in V)$.

\begin{Defn}\label{1-cocycle}
A map $f: G\rightarrow V$ is called a (right) {\bf 1-cocycle} if
$f(xg)=f(x)g+f(g)$  for all $x,g\in G$ holds. Similarly we call $f$ a left 1-cocycle if $f$ satisfies $f(xg)=xf(g)+f(x)$ for all $x,g\in G$.
\end{Defn}

Observe the group algebra $\C V\cong \C^V$ of $V$ becomes a right $\C G$-module setting 
\begin{equation}\nonumber
(\tau.g)(v)=\tau(vg^{-1})\quad \text{ for } \tau\in \C^V, g\in G,  v\in V
\end{equation}
where the action of $G$ on $\C^V$ is denoted by 
$(\tau, g)\mapsto \tau. g$.


We denote $\hat V$ to be the set of linear characters of $V$, i.e. that is of group homomorphisms of  $(V,+)$ into to the multiplicative group $\C^*=\C\setminus\{0\}$ of $\C$. Note that $\hat V$ is contained in $\C ^V$. Indeed under the identification of $\C ^V$ and $\C V$, $\chi\in \hat V$ is mapped to $|V|e_{_{\bar\chi }}$ where $\bar{\chi}$ is the complex conjugate  character of $V$ and $e_{_{\bar{\chi}}}$ is the primitive idempotent of $\C V$ affording $\bar{\chi}$. Since $\{e_{_{{\chi}}}\,|\, \chi\in \hat V\}$ is a basis of $\C V$, we conclude that $\hat V$ is a $\C$-basis of $\C^V$. Moreover one checks immediately:
\begin{Lemma}\label{1.4}
The action $(\tau, g)\mapsto \tau.g$ of $G$ on $\C^V$ permutes $\hat V$. So, for $\chi\in \hat V, g\in G$ we have $\chi.g\in \hat V$.\hfill $\square$
\end{Lemma}

Now let $f: G\rightarrow V$ be a 1-cocycle. Then Jedlitschky showed in [\cite{markus}]:
\begin{Theorem}[]\label{markus}


 
The group algebra $\C^ V$ becomes a monomial $\C G$-module with monomial basis $\hat V$ by setting
$$
\chi g
=\chi(f(g^{-1}))\chi.g 
$$
for $\chi\in \hat V, g\in G.$ \hfill $\square$
\end{Theorem}

 Note that for 1-cocycle $f:G\rightarrow V$ (left or right) we  always have $f(1_G)=0\in V$.

\begin{Lemma}
Let  $f:G\rightarrow V$ be a 1-cocycle (or left 1-cocycle). Then 
$\ker f=\{g\in G\,|\, f(g)=0\in V\}$ is a subgroup of $G$.
\end{Lemma}
\begin{proof}

Let $a,b\in \ker f$, then
$f(ab)=f(a)b+f(b)=0\cdot b+0=0$ proving $ab\in \ker f$. Moreover
$$
0=f(1)=f(aa^{-1})=f(a)a^{-1}+f(a^{-1})=f(a^{-1})
$$
showing $a^{-1}\in \ker f$. The case of a left 1-cocycle $f$ is shown similarly.
\end{proof}

Let $f:G \rightarrow V$ be a 1-cocycle. We define
$$
f^*: \C^ V\rightarrow \C ^G: \tau \mapsto \tau\circ f \quad \text{ for } \tau \in \C^ V. 
$$
\begin{Lemma}[][\cite{markus}]\label{1.8}
Consider $\C^ V$ as a monomial $\C G$-module as defined in \ref{markus}. Then $f^*: \C^ V\rightarrow \C ^G\cong \C G$ is a $\C G$-homomorphism. In particular the image of $f^*$ is isomorphic to a right ideal $\C G$.
\end{Lemma}

Note that by [\cite{markus}] $f^*: \C^V\rightarrow \C^G$ is injective (surjective) if and only if $f: G\rightarrow V$ is surjective (injective). In particular, if $f:G\rightarrow V$ is surjective, the image $f^*$ provides an isomorphism from the $\C G$-module $\C V$ to an right ideal of $\C G$. Indeed we have the following generalisation of this:

\begin{Theorem}\label{1.9}
Let $G$  and $V$ be finite groups, $V$ abelian, and let $f:G\rightarrow V$ be a surjective 1-cocycle. Let $H\leqslant G$ be the kernel of $f$. Then 
$\C V\cong \C^V\cong \Ind_{H}^G {\C}$ as $\C G$-module, where ${\C=}\C_H$ is the trivial $H$-module.
\end{Theorem}
\begin{proof}
By \ref{3.2} it suffices to show that
$f^*(\C^V)\subseteq \C^G$  consists precisely of all functions from $G$ to $\C$
which are constant on the cosets of $H$ in $G$. Let $\tau:V\rightarrow \C$ be a function, $h\in H$ and $g\in G$. Then $f(h)=0$ and
\begin{eqnarray*}
f^*(\tau) (hg)=\tau( f(hg))=\tau(f(h)g+f(g))=\tau(f(g))=f^*(\tau)(g).
\end{eqnarray*}
This shows that $f^*(\tau)$ is constant on the cosets of $H$ in $G$ and hence
$\im f^*=f^*(\C ^V) \leqslant \Ind_{H}^G \C_H $.

For $v\in V$ define $\tau_v: V\rightarrow \C: u\mapsto \delta_{vu}$, (so $\{\tau_v\,|\,v\in V\}$ is a group basis of $\C^ V$). Note that for $g\in G, h\in H,v\in V$ we have 
\begin{eqnarray*}
(f^*(\tau_v))(hg)&=&f^*(\tau_v)(g)=\tau_v(f(g))\\
&=&\begin{cases}
1 &\text{ for } f(g)=v\\ 0 & \text{ for } f(g)\neq v
\end{cases}.
\end{eqnarray*}
So $f^*(\tau_v)$ takes the coset $Hg$ to 1 for $f(g)=v$. Let $a,b\in G$ and suppose that $f(a)=f(b)$. Then
$f(ab^{-1})=f(a)b^{-1}+f(b^{-1})=f(b)b^{-1}+f(b^{-1})=f(bb^{-1})=f(1)=0$ showing that $a$ and $b$ are contained in the same coset of $H$ in $G$. We conclude that $f^*(\tau_v)=0$ on all cosets of $H$ in $G$ different from $Hg$ with $f(g)=v$. This shows that $f^*: \C ^V \rightarrow \Ind_H^G \C_H$ is surjective, by choosing for each $v\in V$ an element $g\in G$ such that $f(g)=v$ using surjectivity of $f$.
\end{proof}

\section{Pattern subgroups}\label{sec2}
Throughout  $q$ is a power of some prime $p$ and $\F_q$ is the field with $q$ elements. The $n$-dimensional $\F_q$-space $\F_q^n$ is the natural module for the general linear group $GL_n(q)$ (acting from right). The {\bf root system} $\Phi$ of $GL_n(q)$ consists of all positions $\{(i,j)\,|\,1\leqslant i, j \leqslant n, i\neq j\}$ of $n\times n$-matrices, and $\Phi=\Phi^+\cup \Phi^-$ with $\Phi^+=\{(i,j)\,|\,1\leqslant i<j \leqslant n\}$, \, $\Phi^-=\{(i,j)\,|\,1\leqslant j<i \leqslant n\}$, (positive and negative roots). The $\F_q$-algebra of $n\times n$-matrices over $\F_q$ is denoted by $M_n(q)$. For $A\in M_n(q)$, $1\leqslant i,j \leqslant n$, let $A_{ij}\in \F_q$ be the entry at position $(i,j)$ in $A$. Thus $A=\sum_{1\leqslant i, j\leqslant n}A_{ij}e_{ij}$, where $e_{ij}$ is the $(i,j)$-th matrix unit. In particular $E=\sum_{1\leqslant i \leqslant n} e_{ii}$ is the identity of $GL_n(q)$.

For $(i,j)\in \Phi, \alpha\in \F_q$, let $x_{ij}(\alpha)=E+\alpha e_{ij}$. Then the {\bf root subgroup} $X_{ij}=\{x_{ij}(\alpha)\,|\,\alpha\in \F_q\}$ is isomorphic to the additive group $(\F_q,+)$.

Recall that for $A\in M_n(\F_q)$, $(i,j)\in \Phi$, $\alpha\in \F_q$, $x_{ij}(\alpha) A$ (respectively $Ax_ {ij}(\alpha)$) is obtained from $A$ by multiplying row $j$ (column $i$) of $A$ by $\alpha$ and adding it to row $i$ (column) $j$.


A subset $J\subseteq \Phi$ is {\bf closed}, if $(i,j),(j,k)\in J$ and $(i,k)\in \Phi$ implies $(i,k)\in J$. Note that if $J\subseteq \Phi^+(\Phi^-)$, the condition $(i,k)\in \Phi$ is automatically satisfied and can be omitted.


For $A\in M_n(\F_q)$, the {\bf support} $\supp(A)$ of $A$ is the set of positions $(i,j), 1\leqslant i, j \leqslant n$ such that $A_{ij}\neq 0$. Thus $A=\sum_{(i,j)\in\supp(A)} A_{ij}e_{ij}$. If $J\subseteq \Phi^- (\Phi^+)$ is closed, then $U_J=\{E+A\,|\, A\in M_n(\F_q), \,\supp(A)\subseteq J\}$ is a subgroup, called {\bf pattern subgroup}, of the $p$-Sylow subgroup $U^-=U=U_n(q)$ ($U^+$) of $GL_n(q)$ of lower (upper) unitriangular matrices. Moreover $U_J$ is generated by the root subgroups $X_{ij}$, $(i,j)\in J$. Indeed it is well known and follows easily from Chevalley's commutator formula (see e.g. [\cite{carter}])  that fixing an arbitrary linear ordering on $J$ and taking every product in this ordering, every element of $U_J$ can be uniquely  written as product:
\begin{equation}\label{2.3}
\prod_{(i,j)\in J} x_{ij} (\alpha_{ij}) \text{ with } \alpha_{ij}\in \F_q.
\end{equation}
Obviously $U=U^-=U_{\Phi^-}$ and $U^+=U_{\Phi^+}$. Inspecting Chevalley's commutator formula one obtains immediately:
\begin{Lemma}\label{2.4}
Let $J\subseteq I\subseteq \Phi^-$ ($\Phi^+$) be closed. Then $U_J$ is a normal subgroup of $U_I$ if and only if the following holds:
If $(i,j), (j,k)\in I$ and $(i,j)$ or $(j,k)$ is in $J$, then $(i,k)\in J$.  \hfill $\square$
\end{Lemma}

\begin{Lemma}\label{2.5}
Let $L,J\subseteq \Phi^-$ be closed, $L\cap J=\emptyset$. Suppose that for $(i,j), (j,k)\in \Phi^-$ such that one of these roots is contained in $L$ the other in $J$ we have always $(i,k)\in J$. Then $K=L\cup J\subseteq \Phi^-$ is closed, $U_L, U_J\leqslant U_K, U_J\trianglelefteq U_K$ and $U_K$ is the semi-direct product of $U_J$ by $U_L$. Let $M=\Ind^{U_K}_{U_L} \C$ and $e$ be a $\C$-vector space generator of the trivial $U_L$-module $\C$. Then $\{eu\,|\,u\in U_J\}$  is a $\C$-basis of $M$ on which $U_J$ acts by multiplication and $U_L$ by conjugation, that is, $eul=eu^l$ for $u\in U_J, l\in U_L$.
\end{Lemma}
\begin{proof}
One checks immediately that $K$ is closed in $\Phi^-$ and that $U_K=U_J\rtimes U_L$.
Obviously $(eu_1)u_2=eu_1u_2$ for $u_1, u_2\in U_J$ and $eul=ell^{-1}ul=eu^l$ for $l\in U_L$ and $u\in U_J$.
\end{proof}

For any subset $J\subseteq \Phi^-$ we set
\begin{equation}\label{2.6}
V_J=\{A \in M_n(q)\,|\, \supp A\subseteq J\}.
\end{equation}

Note that if $J$ is closed then $V_J=\{u-E\,|\, u\in U_J\}$ is the Lie algebra $\Lie(U_J)$ associated with the pattern subgroup $U_J$  and is a (nilpotent) $\F_q$-subalgebra of $M_n(q)$. Moreover, the multiplication rule $e_{ij}e_{kl}=\delta_{jk}e_{il}$ for matrix units implies with \ref{2.4} immediately:
\begin{Lemma}\label{2.7}
Let $J\subseteq  K\subseteq \Phi^-$ be closed. Then $U_J\trianglelefteq U_K$ if and only if $V_J$ is an ideal of $V_K$. \hfill $\square$
\end{Lemma}

Note that in the situation of \ref{2.7} with $U_J\trianglelefteq U_K$, the natural projection map $\tilde \rho: V_K \rightarrow V_{K\setminus J}=\{A\in M_n(q)\,|\,\supp A\subseteq K\setminus J\}$ is an $\F_q$-algebra  homomorphism with kernel $V_J$.


We now apply the results of section \ref{sec1} to the following situation:
\begin{Hypo}\label{2.8}
Let $J, L\subseteq \Phi^-$ be closed, $J\cap L=\emptyset, K=J\cup L$ and $U_J\trianglelefteq U_K$. 
\end{Hypo}
Under this hypothesis  we have: $K\subseteq \Phi^{-}$ is closed and $U_K=U_J\rtimes U_L, V_J\trianglelefteq V_K$ and the natural projection 
\[
\tilde \rho: V_K\rightarrow V_L
\]
is an $\F_q$-algebra homomorphism with kernel $V_J$ and $\tilde \rho|_{{V_L}}=\id_{{V_L}}$. We define $\rho:U_K\rightarrow U_L: u\mapsto \tilde \rho(u-E)+E$ for $u\in U_K$. Obviously $\rho$ is a group epimorphism with kernel $U_J$. Note that $ \rho|_{U_L}=\id_{U_L}$, so $\rho$ is the natural projection of $U_K=U_J\rtimes U_L$ onto $U_L$.

\begin{Lemma}\label{2.9}
Suppose \ref{2.8}. Then $U_K$ acts on $V_J$ as group of $\F_q$-vector space automorphisms, where the action $\circ: V_J\times U_K \rightarrow V_J$ of $U_K$ on $V_J$ being defined by $A\circ g=\rho(g^{-1})Ag$ for $A\in V_J, g\in U_K$.
\end{Lemma}
\begin{proof}
Let $g\in U_K, A\in V_J$. Then
\begin{eqnarray*}
\rho(g^{-1})Ag&=&(\tilde \rho(g^{-1}-E)+E)A ((g-E)+E)\\
&=& A+A(g-E)+\tilde \rho(g^{-1}-E)A+\tilde \rho(g^{-1}-E)A(g-E)
\end{eqnarray*}
is contained in $V_J$, since $g-E, \tilde \rho(g^{-1}-E) \in V_K$ and $V_J$ is an ideal of $V_K$. Since $\rho$ is a group homomorphism the claim follows at once.
\end{proof}

\begin{Lemma}\label{2.10}
Suppose \ref{2.8}. Then the map
\[
f: U_K\rightarrow V_J: g\mapsto \rho(g^{-1})g-E
\]
is an 1-cocycle with kernel $U_L$. Moreover $f|_{U_J}$ is bijective and hence $f$ is in particular surjective.
\end{Lemma}
\begin{proof}
Let $g\in U_K$. Since $\rho(g^{-1})$ and $g$ are contained in $U_K$, we have $f(g)=\rho(g^{-1})g-E\in V_K$ and hence
\begin{eqnarray*}
\tilde \rho(f(g))+E&=&\rho(f(g)+E)=\rho(\rho(g^{-1})g)=\rho(\rho(g^{-1}))\rho(g)\\&=&\rho(g^{-1})\rho(g)=E,
\end{eqnarray*}
since $ \rho(g^{-1})\in U_L$ and $\rho|_{_{U_L}}=\id_{_{U_L}}$. Thus $\tilde \rho(f(g))=0$ and hence $f(g)\in \ker \tilde \rho=V_J$.

Let $x,g\in U_K$. Then
\begin{eqnarray*}
f(x)\circ g+ f(g) &=& \rho(g^{-1})(\rho(x^{-1})x-E)g+\rho(g^{-1})g-E\\
&=&\rho(g^{-1}x^{-1})xg-\rho(g^{-1})g+\rho(g^{-1})g-E\\
&=& \rho(g^{-1}x^{-1})xg-E=f(xg),
\end{eqnarray*}
thus $f$ is a 1-cocycle.

Suppose $0=f(g)=\rho(g^{-1})g-E$, then $\rho(g^{-1})g=E$ and hence $\rho(g)=g$, since $\rho:U_K\rightarrow U_L$ is a group homomorphism.
But $\rho(g)=\tilde \rho(g-E)+E$ and hence $0=f(g)$ is equivalent to $\tilde \rho(g-E)=g-E$, and hence $g-E\in V_L$ that is $g\in U_L$. So $\ker f=U_L$.

Finally for $g\in U_J$, $\rho(g)=E$ and hence $f(g)=g-E$. In particular $f|_{U_J}: U_J\rightarrow V_J: g\mapsto g-E$ is a bijection and thus $f$ is in particular surjective.
\end{proof}

Using \ref{markus} and \ref{1.9} we have shown :
\begin{Theorem}\label{2.11}
Assume \ref{2.8} and let $V=V_J$. Let $f: U_K \rightarrow V_J$ be the 1-cocycle  defined in \ref{2.10}. Then the group algebra $\C V\cong \C^V$ of the finite abelian group $(V,+)$ is a right $\C U_K$-module satisfying the following:
\begin{enumerate}
\item [1)] $U_K$ acts on the $\C$-basis $\hat V=\Hom ((V,+), \C^*)$ monomially, the action of $g\in U_K$ on $\chi\in \hat V$ given by 
\[
\chi g=\chi(f(g^{-1}))\chi.g,
\]
where $\chi.g\in \hat V$ is defined by $\chi.g(A)=\chi(A\circ g^{-1})$ for $A\in V.$
\item [2)] The $\C U_K$-module $\C V\cong \C^V$ is isomorphic to $\Ind_{U_L}^{U_K}\C$, the trivial $\C U_L$-module $\C=\C_{U_L}$ induced to $U_K$.
\item [3)] The restriction $\Res_{U_J}^{U_K}(\C ^V)$ is isomorphic to the right regular $\C U_J$-module.
\item [4)] For $g\in U_L$ we have $f(g^{-1})=0$, hence $\chi(f(g^{-1}))=1$. So $U_L$ acts on $\hat V$ by permutations.
\end{enumerate}
\hfill $\square$
\end{Theorem}


Recall from basic linear algebra, that multiplying $A\in M_n(q)$ from the left (right) by $x_{ij}(\alpha)=E+\alpha e_{ij}$ ($\alpha\in \F_q, 1\leqslant i,j\leqslant n, i\neq j$) produces a matrix obtained from $A$ by adding $\alpha$ times row $j$ (column $i$) of $A$ to row $i$ (column $j$). For any set $J\subseteq \{(i,j)\,|\,1\leqslant i,j\leqslant n\}$ of positions and $V=V_J=\{A\in M_n(q)\,|\,\supp (A)\subseteq J\}$ we denote the projection which takes $A\in M_n(q)$ to $\sum_{(i,j)\in J} A_{ij}e_{ij}\in V_J$
by $\pi _{\!_J}$. Thus $\pi_{\!_J}$ is the natural $\F_q$-vector space projection of $M_n(q)$ onto $V_J$. Adding $\alpha$ times ($\alpha\in \F_q$) row (column) $i$ to row (column) $j$ in a matrix $A\in M_n(q)$ ($1\leqslant i, j\leqslant n, i\neq j$) and following up with $\pi_{\!_J}$, (so obtaining a matrix in $V_J$) is called {\bf truncated row (column) operation}, from row (column) $i$ to row  (column) $j$ along $V_J$.


We return to the setting of \ref{2.8}. We want to exhibit explicit formulas for the action of $U_K$ on $\hat V$, where again we set $V=V_J=\{A\in M_n(q)\,|\, \supp (A)\subseteq J\}$.


First we describe $\hat V$. Recall that $V$ is an $\F_q$-vector space with $\F_q$-basis $\{e_{ij}\,|\, (i,j)\in J\}$. Let $\{\epsilon_{ij}\,|\,(i,j)\in J\}$ be the dual $\F_q$-basis of the dual space $V^*=\Hom_{\F_q}(V, \F_q)$. Thus $\epsilon_{ij}$ maps $A\in V=V_J$ to its $(i,j)$-th coordinate $A_{ij}\in \F_q$ for $(i,j)\in J$, and every linear function $\chi^*\in V^*$ on $V$ is a unique linear combination $\sum_{(i,j)\in J} B_{ij}\epsilon_{ij}=\chi^*$ with $B_{ij}\in \F_q$. Let $B=\sum_{(i,j)\in J} B_{ij}e_{ij}\in M_n(q)$ and denote $\chi^*$ by $\chi_{_B}^*$, so $\{\chi_{_B}^*\,|\,B\in V\}=V^*$.

We choose once for all a nontrivial linear character
\begin{equation}\label{2.13}
\theta: (\F_q,+)\rightarrow \C^*
\end{equation}
of the additive group of the field $\F_q$ into the multiplicative group $\C^*=\C\setminus \{0\}$.
The following result is easily checked by direct calculation:
\begin{Lemma}\label{2.14}
For $\chi^*=\chi_{_B}^*\in V^*, B\in V$, let $\chi_{_B}:V\rightarrow \C^*$ be the composite map $\theta \circ \chi_{_B}^*$. Then $\chi_{_B}\in \hat V$ and $\chi_{_B}\neq \chi_{_C}$ for $B, C\in V, B\neq C$. In particular, $\hat V=\{\chi_{_B}\,|\,B\in V\}$. \hfill $\square$
\end{Lemma}

For $B\in V, g\in U_K$, \ref{1.4} implies that $\chi_{_B}.g=\chi_{_C}$ for $C\in V$ satisfying $(\chi_{_B}^*.g) (e_{st})=C_{st}$, where $(s,t)\in J$. So  in order to determine $C$,  we have to evaluate $\chi_{_B}^*.g$ at matrix unites $e_{st}$ with $(s,t)\in J$. For convenience we denote $C$ by $B.g$. Then $\chi_{_B}^*.g=\chi^*_{_{B.g}}$.

\begin{Lemma}\label{2.15}
Suppose \ref{2.8} and let $B\in V=V_J, g\in U_K$. Then $B.g=\pi_{\!_J}((\rho(g))^t B g^{-t})$, where on the right hand side we have ordinary matrix multiplication and $A^t$ denotes the transpose of $A\in M_n(q)$ and $g^{-t}=(g^{-1})^t=(g^t)^{-1}$.
\end{Lemma}
\begin{proof}
Since $U_K$ is generated by root subgroups $X_{kl}$ with $(k,l)\in K=J\cup L$, we may assume $g=x_{kl}(\alpha)$ with $\alpha\in \F_q$. Let $(i,j), (s,t)\in J$, then 
\begin{equation}\label{2.16}
\beta_{ij}:=(\epsilon_{ij}.x_{kl}(\alpha))(e_{st})=\epsilon_{ij}(e_{st}\circ x_{kl}(-\alpha))=\epsilon_{ij}(\rho(x_{kl}(\alpha))e_{st}x_{kl}(-\alpha))
\end{equation}

{\bf 1. Case:} Suppose $(k,l)\in L.$ Then by \ref{2.8}, definition of $\rho$,
\begin{eqnarray*}
\beta_{ij}&=&\epsilon_{ij}(x_{kl}(\alpha)e_{st}x_{kl}(-\alpha))= \epsilon_{ij}((E+\alpha e_{kl})e_{st}(E-\alpha e_{kl}))\\
&=& \epsilon_{ij}(e_{st}-\alpha e_{st}e_{kl}+\alpha e_{kl}e_{st}-\alpha^2 e_{kl}e_{st}e_{kl})
\end{eqnarray*}
Since $(k,l), (s,t)\in \Phi^-$, we have $l<k, t<s$ and hence $t=k$ implies $l<s$. Thus $e_{kl}e_{st}e_{kl}=0$ and
\begin{eqnarray}\label{2.1777}
\beta_{ij}&=& \epsilon_{ij}(e_{st}-\alpha e_{st}e_{kl}+\alpha e_{kl}e_{st})\nonumber\\
&=&\begin{cases}
 \epsilon_{ij}(e_{st}-\alpha e_{sl})=\delta_{is}\delta_{jt}-\alpha \delta_{is}\delta_{jl} & 
 \text{ if } t=k
 \\
 \epsilon_{ij}(e_{st}+\alpha e_{kt})=\delta_{is}\delta_{jt}+\alpha \delta_{ik}\delta_{jt} & 
 \text{ if } s=l
 \\
 \epsilon_{ij}(e_{st})=\delta_{is}\delta_{jt} & 
 \text{ otherwise}.
\end{cases}
\end{eqnarray}

Now $\chi_{_B}^*=\sum_{(i,j)\in J}B_{ij}\epsilon_{ij}$ and hence by (\ref{2.16}):
\begin{eqnarray*}
C_{st}&=&\chi_{_C}^*(e_{st})=(\chi_{_B}^*.x_{kl}(\alpha))(e_{st})=
\sum_{(i,j)\in J}B_{ij} (\epsilon_{ij}(e_{st}\circ x_{kl}(-\alpha)))\\&=& \sum_{(i,j)\in J} B_{ij}\beta_{ij}
\stackrel{(\ref{2.1777})}=\begin{cases}
B_{st}-\alpha B_{sl} &\text{ for } t=k\\
B_{st}+\alpha B_{kt} & \text{ for }s=l\\
B_{st}& \text{ otherwise}.
\end{cases}
\end{eqnarray*} 
Thus we obtain matrix $C=B.x_{kl}(\alpha)$ from $B$ by adding $-\alpha$ times column $l$ to column $k$, adding $\alpha$ times row $k$ to row $l$ in $B$ and projecting the resulting matrix into $V=V_J$. By basic linear algebra this is indeed  $\pi_{\!_J}(x_{lk}(\alpha)Bx_{lk}(-\alpha))$ as desired.

{\bf 2. Case:} Now suppose $(k,l)\in J$. Then
\begin{eqnarray*}
\beta_{ij}&=&(\epsilon_{ij}.x_{kl}(\alpha))(e_{st})=\epsilon_{ij}(e_{st}(E-\alpha e_{kl}))\\
&=&\begin{cases}
\epsilon_{ij} (e_{st}-\alpha e_{sl})=\delta_{is}\delta_{jt}-\alpha \delta_{is}\delta_{jl} &\text{ if } t=k\\
\epsilon_{ij} (e_{st})=\delta_{is}\delta_{jt} & \text{ otherwise}
\end{cases}
\end{eqnarray*}
and we conclude that $C=B.x_{kl}(\alpha)$ is obtained from $B$ by adding $-\alpha$ times column $l$ to column $k$ and projecting the resulting matrix into $V$ by $\pi_{\!_J}$. But this is again $\pi_{\!_J}(Bx_{lk}(-\alpha))=\pi_{\!_J}((\rho(x_{kl}(\alpha)))^tBx_{kl}(-\alpha)^t)$, since $\rho(x_{kl}(\alpha))=E$  for $(k,l)\in J$.
\end{proof}

\begin{Defn}\label{2.177}
Let $1\leqslant j<i\leqslant n$, then $-\epsilon_{ij}(A)=-A_{ij}$ for $A\in V_J$, and hence $\theta\circ (-\epsilon_{ij})(A)=\overline{\theta(A_{ij})}$, where $\bar z$ denotes the complex conjugate of $z\in \C$. From this it follows immediately, that $\chi_{_{-A}}=\overline{\chi_{_A}}$, the complex conjugate character to $\chi_{_A}$ for $A\in V_J$. So for $A\in V_J$, we define $e_{_A}=q^{-|J|} \sum _{B\in V_J} \overline{\chi_{_A}(B)} B$. Thus $q^{-|J|} \chi_{_{-A}}\in \C^ V$ is identified under (\ref{1.1}) with the primitive idempotent  $e_{_A}\in \C (V_J,+)$ affording the linear character $\overline{\chi_{_A}}=\chi_{_{-A}}$. In order to distinguish the idempotents $e_{_A}\in \C (V,+)$ from idempotents in the group algebra $\C U_J$, we call $e_{_A}$, $A\in V_J$ ``lidempotent'' indicating that it belongs to the Lie algebra of $U_J$ rather than $U_J$. Thus $\cE_{J}=\{e_{_A}\,|\,A\in V_J\}$ is called {\bf lidempotent basis} of $\C(V_J,+)$.
\end{Defn}

Combining \ref{2.11} and \ref{2.15} we have shown:
\begin{Theorem}\label{2.17}
Suppose \ref{2.8} and let $f: U_K\rightarrow V=V_J$ be the  1-cocycle defined in \ref{2.10}. Let $A\in V=V_J$ and $g=x_{ij}(\alpha)\in U_K, (k,l)\in K=J\cup L, \alpha\in \F_q$. Then $\chi_{_A}.x_{kl}(\alpha)=\chi_{_B}$, where $B\in V$ is obtained from $A$ by the truncated column operation along $V$  from column $l$ to column $k$ using the factor $-\alpha$, if $(k,l)\in J$. If $(k,l)\in L$, then $B$ is obtained from $A$ by  combined truncated row and column operation along $V$ from column $l$ to column $k$, row $k$ to row $l$ with factors $-\alpha$ and $\alpha$ respectively. Moreover
\begin{equation}\label{2.18}
\chi_{_A} g=
\begin{cases}
\theta(-\alpha A_{ij})\chi_{_B} & \text{ for } (i,j)\in J\\
\chi_{_B} & \text{ for } (i,j)\in L.
\end{cases}
\end{equation}
As a consequence, we obtain a monomial action of $U_K$ on the lidempotent basis $\cE_J$ of $\C (V_J,+)$ (defined in \ref{2.177}) as follows:
\begin{equation}\label{2.199}
e_{_A} g=
\begin{cases}
\theta(\alpha A_{ij})e_{_B} & \text{ for } (i,j)\in J\\
e_{_B} & \text{ for } (i,j)\in L.
\end{cases}
\end{equation}
\end{Theorem}
\begin{proof}
Everything is already shown except (\ref{2.18}) and (\ref{2.199}). By \ref{2.11}
\[
\chi_{_A}g=\chi_{_A}(f(g^{-1}))\chi_{_B}=\theta(\chi_{_A}^*(f(g^{-1})))\chi_{_B}.
\]
For $(i,j)\in L, f(g^{-1})=0$ and the claim follows (comp. \ref{2.11} 4)). So let $(i,j)\in J$. Then using \ref{2.10} we obtain $\chi_{_A} ^*(f(g^{-1}))=\chi_{_A}^*(g^{-1}-E)=\chi_{_A}^* (-\alpha e_{ij})=-\alpha A_{ij}$ and (\ref{2.18}) follows.   Then by the identification given in \ref{2.177},  (\ref{2.199}) follows .
\end{proof}

\begin{Remark}\label{2.19}
The special case $L=\emptyset, J=K$ in \ref{2.8} yields a supercharacter theory of the pattern subgroup $U_J$ of $U$. In this case the action of $U_J$ on $\hat V_J$ can entirely be described by (ordered) sequences of truncated column operations.
\end{Remark}

\begin{Example}\label{ExLi}
We illustrate the lidempotent   $e_{_A}\in \C V_J$ by a triangle, omitting superfluous zeros in the upper half of matrix $A\in V$ and indicating positions not belonging to $J$ by "$\ddagger$" . Let $n=6$ and $V=V_J$ where $J$ is obtained by taking out column 2 and column 4  from $\Phi^-$. Let $A=\sum_{(i,j)\in J}A_{ij}e_{ij}\in V$, $A_{ij}\in \F_q$, and let $\alpha\in\F_q$. Then by (\ref{2.199}) we have

%
%
%
%
%
%
%

\begin{center}
\begin{picture}(180, 95)
\put(-75,45){$e_{_A}x_{53}(\alpha)=\theta(\alpha A_{53})$}
\put(30,0){\line(1,0){100}}
\put(30,0){\line(0,1){100}}
\put(30,100){\line(1,-1){100}}
\put(35,78){0}
\put(32,63){\small$A_{21}$}
\put(32,48){\small$A_{31}$}
\put(32,33){\small $A_{41}$}
\put(32,18){\small $A_{51}$}
\put(32,3){\small $A_{61}$}
\put(50,63){0}
\put(50,33){$\ddagger$}
\put(50,48){$\ddagger$}
\put(50,18){$\ddagger$}
\put(50,3){$\ddagger$}
\put(60,3){\small $A_{63}$}
\put(60,33){\small $A_{43}$}
\put(60,18){\small $A_{53}$}
\put(65,48){0}

\put(80,33){0}
\put(80,18){$\ddagger$}
\put(80,3){$\ddagger$}

\put(95,18){$0$}
\put(92,3){$\beta$}

\put(110,3){$0$}

\put(150,45) {where $\beta=A_{65}-\alpha A_{63}$.}

\end{picture}
\end{center}
\end{Example}

\begin{Notation}\label{orbit}
In the situation of \ref{2.8} we denote the $U_K$-orbit of $\cE_J$ containing $e_{_A}\in \cE_J, A\in V_J$ by $\cO_A^J$. In the special case $L=\emptyset, J=K=\Phi^-$, we drop the sub- and super index $J$, that is $V=V_{\Phi^-}, \cE=\cE_{\Phi^-}$ and $\cO_{_A}=\cO_{_A}^{\Phi^-}$ for $A\in V$.
\end{Notation}

\section{$\lambda$-flags} \label{sec3}
For $n\in \N$ we denote the symmetric group on $n$ letters by $\fS_n$. A {\bf composition} $\lambda$ of $n$, denoted by $\lambda\vDash n$  is a finite sequence $\lambda=(\lambda_1, \ldots, \lambda_k)$ of non negative integers $\lambda_i\in \N$, $1\leqslant i \leqslant k$, whose sum is $n$. 
The {\bf (Young) diagram} of a composition $\lambda$ is the subset
$$[\lambda]=\{(i,j)\,|\,1\leqslant j\leqslant \lambda_i\text{ and }i\geqslant 1\}$$ of $\N \times \N.$ If $(i,j)\in [\lambda]$, then $(i,j)$ is called a {\bf node} of $\lambda$. We represent the diagram as an array of boxes in the plane.
Suppose $\lambda$ is a composition of $n$. A {\bf $\lambda-$tableau} is a bijection
$$\t : [\lambda]\rightarrow \{1,2,\cdots,n\};$$ we say that $\t$ has  shape $\lambda$ and write Shape$(\t)=\lambda.$
We represent a tableau $\t$ by a labeled diagram replacing every node $(i,j)$ in $[\lambda]$  by it's image in $\{1,2,\cdots,n\}$ under the map $\t.$ For example,
$$\t_1=\begin{tabular}{|c|c|c|}\cline{1-2}1&2\\\hline 3&4&5\\\hline 6&7\\\cline{1-2}\end{tabular}
\text{ and }\t_2=\begin{tabular}{|c|c|c|}\cline{1-2}1&4\\\hline 2&5&7\\\hline 3&6\\\cline{1-2}\end{tabular}$$ are two $\lambda$-tableaux, where $\lambda=(2,3,2).$
The symmetric group $\mathfrak S_n$ acts on the set of $\lambda$-tableaux from the right by permuting the integers $1,2,3,\dots,n.$
In the Example above, $$\t_1(2\,4\,5\,7\,6\,3)=\t_2.$$

A $\lambda$-tableau is {\bf row standard}, if the entries increase along rows from left to right.  The set of $\lambda$-tableaux (resp. row standard $\lambda$-tableuax) is denoted by $ \Tabl(\lambda)$, (resp. $\RStd(\lambda)$). The {\bf initial $\lambda$-tableau} $\t^\lambda$ is the $\lambda$-tableau, where the numbers $1, \ldots, n$ are inserted in order along the rows down row by row. For $\s\in \Tabl(\lambda)$, we denote the permutation $w\in \fS_n$ with $\t^\lambda w=\s$ by $\d(\s)$. For $1\leqslant i \leqslant n$, $\s\in \Tabl(\lambda)$, $\row_\s(i)=m$ if the $m$-th row counted from the top in $\s$ contains $i$.

For $\lambda\vDash n$, the {\bf row stabilizer} of $\t^\lambda$ in $\fS_n$ is the {\bf standard Young subgroup} $\fS_\lambda=\fS_{\{1, \ldots, \lambda_1\}}\times \fS_{\{\lambda_1+1, \ldots, \lambda_1+\lambda_2\}}\times \cdots \cong \fS_{\lambda_1}\times \cdots \times \fS_{\lambda_k}\leqslant \fS_n$, $\lambda=(\lambda_1, \ldots \lambda_k)$. For any $\s\in \Tabl(\lambda)$, the row stabilizer of $\s$ is then $\d(\s)^{-1}\fS_\lambda \d(\s)$. The set
\begin{equation}\label{3.1}
\cD_\lambda=\{\d(\s)\,|\, \s\in \RStd(\lambda)\}
\end{equation}
is a set of right coset representatives of $\fS_\lambda$ in $\fS_n$, called {\bf distinguished coset representatives,} (see e.g. [\cite{carter}]).


We identify $\fS_n$ with the Weyl group of $GL_n(q)$, i.e. with the set of permutation matrices in $GL_n(q)$. Thus $w\in \fS_n$ acts in the natural basis $e_1, \ldots, e_n$ of $\F_q^n$ by $e_{i}w=e_{iw}$. Moreover, for $1\leqslant i,j \leqslant n, \alpha\in \F_q$, one checks directly that: $w e_{ij}=e_{iw^{-1},\,j}\text{ and } e_{ij} w=e_{i,\,jw}
$. Hence
\begin{equation}
 x_{ij}(\alpha)^w=w^{-1}x_{ij}(\alpha)w=w^{-1}(E+\alpha e_{ij})w=E+\alpha w^{-1} e_{ij}w=x_{iw,jw}(\alpha).
\end{equation} 


For $\lambda=(\lambda_1, \cdots, \lambda_k)\vDash n$ we define a {\bf$\lambda$-flag} in $V=\F_q^n$ to be a sequence
\[
V=V_0> V_1>\cdots >V_{k-1}\geqslant V_k=(0)
\]
of subspaces $V_i$ of $V$ ($i=0, \ldots,k$) such that $\lambda_i=\dim V_{i-1}-\dim V_i$ holds. We denote the set of $\lambda$-flags in $V$ by $\cF(\lambda)$. Obviously $GL_n(q)$ acts on $\cF(\lambda)$ by left- and right multiplication transitively on both sides.


We set $\Lambda_i=\lambda_1+\lambda_2+\cdots+\lambda_i$ for $i=0,1,\ldots,k$. For the natural basis $e_1, \cdots, e_n$ of $V=\F_q^n$ set $V_i=\langle e_{_{\Lambda_i+1}}, e_{_{\Lambda_i+2}}, \cdots ,e_{n}\rangle$. Then the $\lambda$-flag 
$V=V_0> V_1> \cdots > V_{k-1}> V_k=(0)$
is called standard $\lambda$-flag and its stabilizer in $GL_n(q)$ acting from the right on $\cF(\lambda)$ is the {\bf standard parabolic subgroup} $P_\lambda$. By general theory, the linear permutation module $\C \cF(\lambda)$ spanned by $\cF(\lambda)$  under the right action of $GL_n(q)$ is isomorphic to the trivial $\C P_\lambda$-module induced to $GL_n(q)$, denoted by $\Ind_{P_\lambda}^{GL_n(q)} \C$.

\begin{Defn}\label{3.3}
Let $\lambda\vDash n$. Then the subsets
\begin{eqnarray*}
J_\lambda&=&\{(i,j)\in \Phi\,|\,\row_{\t^\lambda}(i)\leqslant \row_{\t^\lambda}(j)\}\\
J_\lambda^=&=&\{(i,j)\in \Phi\,|\,\row_{\t^\lambda}(i)= \row_{\t^\lambda}(j)\}\\
J_\lambda^<&=&\{(i,j)\in \Phi\,|\,\row_{\t^\lambda}(i)<\row_{\t^\lambda}(j)\}
\end{eqnarray*}
are closed subsets of $\Phi$ and $P_\lambda=\langle B^+, X_{ij}\,|\, (i,j)\in J_\lambda \rangle$ is the standard parabolic subgroup  containing the upper Borel subgroup $B^+$ of invertible upper triangular matrices. Moreover $L_\lambda=\langle T, X_{ij}\,|\, (i,j)\in J_\lambda^= \rangle$ is a Levi subgroup of $P_\lambda$ ({\bf standard Levi subgroup}) isomorphic to $GL_{\lambda_1}(q)\times \cdots GL_{\lambda_k}(q)$, $\lambda=(\lambda_1,\ldots, \lambda_k)\vDash n$, and $U_\lambda=U_{J_\lambda^<}\leqslant U^+$ is the unipotent radical of $P_\lambda$. In particular, $U_\lambda  \unlhd P_\lambda$, $U_\lambda\cap L_\lambda=(1)$ and $P_\lambda=U_\lambda L_\lambda=U_\lambda\rtimes \ L_\lambda$.
\end{Defn}

For later use we set $U_{\lambda}^-=U_{J_\lambda^>}$ where $J_\lambda^>=\{(i,j)\in \Phi\,|\,\row_{\t^\lambda}(i)>\row_{\t^\lambda}(j)\}$. Then $P_\lambda^-=L_\lambda U_\lambda^-$ is the standard parabolic subgroup containing $B^-$, the Borel subgroup of all lower triangular invertible matrices.


Let $F: V=V_0>V_1>\cdots >V_{k-1}> V_k=(0)$ be a $\lambda$-flag. A basis $\{v_1,\cdots,v_n\}$ of $V$ is {\bf $F$-adapted}, if $v_{_{\Lambda_i+1}}, \cdots,v_n$ spans $V_i$ ($i=0,\ldots,k-1$). Writing each $v_i$ (uniquely) as linear combination $v_i=\sum_{j=1}^n \beta_{ij}e_j$, we obtain an invertible matrix $B\in M_n(q)$ with $B_{ij}=\beta_{ij}$, whose row vectors are $v_i$. We illustrate this as follows:
\begin{equation}\label{compartment}
\hspace{-6cm}B=\left(
\begin{picture}(100,55)
\put(10,-40) { {compartment \,\, $k$}}
\multiput(-8,-20)(10,0){16}{\line(1,0){5}}

\put(110,-10) {\small $\bigg\}$ basis of $V_{k-2}$ mod $V_{k-1}(\cong \F_q^{\lambda_{k-1}})$}
\put(110,-36) {\small $\bigg\}$ basis of $V_{k-1}(\cong \F_q^{\lambda_k})$}

\multiput(-8,5)(10,0){16}{\line(1,0){5}}
\put(8,-10){{compartment $k-1$}}

\put(8,40){{compartment \quad $1$}}
\multiput(-8,30)(10,0){16}{\line(1,0){5}}
\put(110,40) {\small $\bigg\}$ basis of $V=V_0$ mod $V_{1}(\cong \F_q^{\lambda_{1}})$}

\put(43,12){\large $\vdots$}
\put(133,12){\large $\vdots$}
\end{picture}
\right)
\end{equation}
Let $1\leqslant i\leqslant n$. If row $i$ lies in the $r$-th compartment of $B$ for some $r\in \{1,2,\ldots,k\}$, then we denote: $\comp_{_B}(i)=r$.
Using basic linear algebra  the following is easy to see:
%
%

\begin{Lemma}\label{3.6}
$P_\lambda$ acts by left multiplication on the set of $F$-adapted bases $B\in GL_n(q)$ of $V$ transitively, inducing a bijection:
$$
\{\text{right } P_\lambda\text{-cosets in }GL_n(q)\}=\{\text{left }P_\lambda \text{-orbits on }GL_n(q)\} \longleftrightarrow \cF_\lambda=\{\lambda\text{-flags in }V=\F_q^n\}.
$$ \hfill $\square$
\vspace{-0.99cm}
$$\hspace{5.4cm}\text{\scriptsize bij.}$$ \end{Lemma}

Obviously the sets of the left $P_\lambda$-orbits on $GL_n(q)$ and the right cosets $P_\lambda g$ ($g\in GL_n(q)$) are the same. The next result gives a detailed description of this. Recall that $U=U^-$ is the lower unitriangular group. For the next result see Cf. [\cite{jamesbook}, Theorem 7.5].

\begin{Prop}\label{3.7}
Let $\lambda\vDash n$. Then $\cD_\lambda$ is a set of $P_\lambda$-$U$ double coset representatives in $GL_n(q)$. Moreover for $d\in \cD_\lambda$, $P_\lambda dU=\bigcup_{u\in (U_\lambda^-)^d\cap U}P_\lambda d u$ is a decomposition of $P_\lambda d U$ into right $P_\lambda$-cosets in $GL_n(q)$. Thus  $\{du\,|\,d\in \cD_\lambda, u\in (U_\lambda^-)^d\cap U\}$ is a set of   right coset representatives of $P_\lambda$ in $GL_n(q)$.
\end{Prop}


Since $U_\lambda^-=\langle X_{ij}\,|\,i,j)\in \Phi,\row_{\t^\lambda}(i)>\row_{\t^\lambda}(j)\rangle$ and for $1\leqslant i\leqslant n,$ $i$ occupies the position in $\s=\t^\lambda d$ which is occupied by $id^{-1}$ in $\t^\lambda$, we have:
\begin{Lemma}\label{3.8}
Let $\lambda\vDash n, d\in \cD_\lambda, \s=\t^\lambda d$. Then
$
(U_\lambda^-)^d\cap U=U_J, \text{ where } J= \{(i,j)\in \Phi^-| \row_\s(i)>\row_\s(j)\}. $
\hfill $\square$
\end{Lemma}

Let $\lambda, d$ and $\s$ as above. Next we shall describe the matrices in the subset $d U_J$ of $GL_n(q)$. Let $u\in U_J$. We consider the row vectors of $du\in GL_n(q)$ as an $F$-adapted basis of some $\lambda$-flag $F$. By \ref{3.8} $u_{ij}\neq0$ implies 
\begin{equation}\label{3.9}
(i=j\text{ and } u_{ii}=1)\quad \text{ or } \quad (i>j\text{ and }\row_\s(i)>\row_\s(j)).
\end{equation}
Note that the left action of $d\in \fS_n$ on $u$ is just permuting the rows of $u$. Thus $(du)_{kr}=u_{kd,r}$ for all $(k,r)\in \Phi$.
So $(du)_{kl}\neq0$ implies by (\ref{3.9})
\[
(kd=r\text{ and } (du)_{k,kd}=1)\quad \text{ or } \quad (kd>r\text{ and }\row_\s(kd)>\row_\s(r)).
\]
Therefore one can see immediately that  the values on the positions $(k,kd)$ in $du$ are 1's and these are  the most right hand side nonzero entries in each row $k\in \{1, \ldots, n\}$. Now assume $(du)_{kr}\neq0$ and $kd>r$. Recall that for $1\leqslant i\leqslant n$, $i$ occupies the position in $\s$ which is occupied by $id^{-1}$ in $\t^\lambda$. Thus $\row_\s(kd)>\row_\s(r)$ if and only if $\row_{\t^\lambda}(k)>\row_{\t^\lambda}(rd^{-1})$ and hence the position $(k,r)$ lies in a lower compartment than the position $(rd^{-1},r)$, on which sits  a  one. 
So we have shown: 
\begin{Theorem}\label{xlambda}
For $\lambda\vDash n$ we denote
\[
\fX_\lambda=\{du\,|\,d\in \cD_\lambda, u\in (U_\lambda^-)^d\cap U\},
\]
called {\bf $\lambda$-normal} matrices. Then any $A\in \fX_\lambda$  satisfies:
\begin{enumerate}
\item [1)]  There exists $d\in \cD_\lambda$ such that $A_{i, id}=1$ is the last nonzero entry in row $i$ of $A$. We call $A_{i, id}=1$ {\bf ``the last one''}, for $1\leqslant i \leqslant n$.  {Define $\tab A=\t^\lambda d\in \RStd(\lambda)$.}  
\item [2)] $A_{r, id}\neq0$ for $1\leqslant i \leqslant n$ implies $r=i$ (and then $A_{r, id}=1$) or $\comp_{_A}(r)>\comp_{_A}(i)$, that is, the compartment  containing row $r$ is lower than the compartment containing row $i$. \hfill $\square$
\end{enumerate}
\end{Theorem}

\begin{Remark}\label{3.11}
By \ref{3.7}, $\fX_\lambda$ is a set of    right coset representatives of $P_\lambda$ in $GL_n(q)$. By \ref{3.6}, (more precisely, using suitable the row operations ,  each  left $P_\lambda$-orbit $P_\lambda A\subseteq GL_n(q)$ contains precisely one row reduced form ({\bf normal form}) $A^\lambda\in P_\lambda A\cap \fX_\lambda$. So from now on, each $\lambda$-flag can be  denoted by an unique element in the set $\fX_\lambda$.
\end{Remark}

\begin{Defn}\label{circ}
Let $\lambda\vDash n$. Define a right action of $GL_n(q)$ on $\fX_\lambda$ by
\[
\bullet: (\fX_\lambda, GL_n(q))\rightarrow \fX_\lambda: (A, g)\mapsto A\bullet g=(Ag)^\lambda\quad \text{ for }A\in \fX_\lambda, \,  g\in GL_n(q).
\]
 \end{Defn}
\begin{Lemma}\label{3.13}
Let $\C \fX_\lambda$ be the vector space over $\C$ with basis $\fX_\lambda$. Then under the ``$\bullet$''-action, $\C \fX_\lambda$ becomes a permutation module isomorphic to $M(\lambda):=\Ind_{P_\lambda}^{GL_n(q)} \C$, where the isomorphism is given by $A\in \fX_\lambda \mapsto \overline {P_\lambda} A$, with $\overline { P_\lambda}=\sum_{h\in P_\lambda}h$. \hfill $\square$
\end{Lemma}

\section{The $\s$-components of $\Res_{\text{\tiny$U$}}^{\text{\tiny$GL_n(q)$}} M(\lambda)$}\label{sec4}
Let $\lambda\vDash n$ and $M(\lambda)=\Ind_{P_\lambda}^{GL_n(q)} \C\cong \C (\overline { P_\lambda} \fX_\lambda)$ be as defined  in \ref{3.13}. By \ref{3.7}, $\cD_\lambda
$ is a set of $P_\lambda$-$U$ double coset representatives. 
We apply Mackey decomposition to obtain a first decomposition:
\begin{Prop}\label{mackey}
Let $\lambda\vDash n$. Then
\begin{enumerate}
\item [1)]$\Res_{\text{\tiny$U$}}^{\text{\tiny$GL_n(q)$}} M(\lambda)=\bigoplus\nolimits_{d\in \cD_\lambda}\Ind_{\text{\tiny$P_\lambda^d \cap U$}}^{\text{\tiny$U$}} \C \overline{P_\lambda} d.$
\item[2)] Let $d\in \cD_\lambda$ and $\s=\t^\lambda d \in \RStd(\lambda)$. Then the $U$-permutation module $M_\s=\Ind_{\text{\tiny$P_\lambda^d \cap U$}}^{\text{\tiny$U$}} \C \overline{P_\lambda} d$ has a $\C$-basis $\{\overline{P_\lambda} du \,|\, u\in (U_\lambda^-)^d\cap U\}.$ \hfill $\square$
\end{enumerate}
\end{Prop}

With $\s\in \RStd(\lambda)$, 
we call $M_\s$ introduced in \ref{mackey} {\bf $\s$-component} of $M(\lambda)$ (in  [\cite{guo}] called $\s$-batch). Under the identification given in \ref{3.13}, its basis is given by  $\fX_\s:=\{du\,|\, u\in (U_\lambda^-)^d\cap U\}$, where $d=\d(\s)\in \cD_\lambda$.  For such an $u$, the matrix $du\in M_n(q)$ is obtained from $u$ by reordering the rows $(r_1,\ldots,r_n)$ of the matrix $u$ to $(r_{1 d},\ldots,r_{n d})$. Note that $1d, \ldots, nd$ are precisely the entries of $\s\in \RStd(\lambda)$ from left to right in the rows of $\s$  and going the rows from top down. In the following we encode $u$ and $du$ in one matrix as follows: We relabel in matrix $du$ the rows top down by  ($1d, \ldots, nd$). With these new row labels the entry at position $(i,j)$ of $du$ coincides  with $u_{ij}\in \F_q$. Moreover note that the rows of $du$ then are automatically divided into $\s$-compartments. So $u$ may be recovered immediately from $du$ by reordering rows of $du$ in its natural order.

\begin{Example}\label{4.2-2}
 Let $\lambda=(2,2,2)\vdash 6$ and $\s=$\begin{tabular}{|c|c|}\hline
1&3\\\hline
2&4\\\hline
5&6\\ \hline
\end{tabular}\,. Then under the new labeling we have
\[
\fX_\s=\left\{\begin{tabular}{c}1\\3\\ 2\\4\\ 5\\6\end{tabular}\left(\begin{tabular}{cccccc}
1&\\
0&0&1&&&\\\hline
$*$&1&\\
$*$&$0$&$*$&1&&\\\hline
$*$&$*$&$*$&$*$&1&\\
$*$&$*$&$*$&$*$&0&$1$
\end{tabular}\right), \quad*\in \F_q\right\}
\]
\end{Example}


Recall that $U=U^-=\langle X_{ij}\,|\, (i,j)\in \Phi^{-} \rangle$. Next we want to describe some pattern subgroups in $U^-$.
Let $\s\in \RStd(\lambda), d=\d(\s)\in \cD_\lambda$.  Recall that  $P_\lambda$ is generated by the torus $T$ and root subgroups $X_{ij}$ with $(i,j)\in J_\lambda$ defined in \ref{3.3}. Thus $X_{ij}\subseteq P_\lambda^d$ if and only if $(id^{-1},jd^{-1})\in J_\lambda$ that is $\row_{\t^\lambda}(id^{-1})\leqslant \row_{\t^\lambda} (jd^{-1})$ or equivalently, if and only if $\row_\s (i)\leqslant \row_\s (j)$. This proves the first three claims in the  the following result observing in addition, that $X_{ij}\subseteq U$ if and only if $j < i$:
%
%

\begin{Lemma}\label{4.3}
Let $\lambda\vDash n, \s\in \RStd(\lambda), d=\d(\s)\in \cD_\lambda$. Then:
\begin{enumerate}
\item [1)]$P_\lambda^d\cap U=\langle X_{ij}\,|\, 1\leqslant j<i\leqslant n, \row_\s(i)\leqslant \row_\s(j)\rangle$

\item [2)]$L_\lambda^d\cap U=\langle X_{ij}\,|\, 1\leqslant j<i\leqslant n, \row_\s(i)=\row_\s(j)\rangle=(L_\lambda\cap U)^d$

\item [3)]$U_\lambda^d\cap U=\langle X_{ij}\,|\, 1\leqslant j<i\leqslant n, \row_\s(i)<\row_\s(j)\rangle$

\item [4)]
$U^d\cap U=\langle X_{ij}\,|\, 1\leqslant j<i\leqslant n, \row_\s(i)\geqslant\row_\s(j)\rangle$
\end{enumerate}
\end{Lemma}
\begin{proof}

4) Let $X_{ij}\subseteq U^d\cap U, (i,j)\in \Phi$. Since $X_{ij}\subseteq U$, we have $i>j$. Now $X_{ij}\subseteq U^d$ if and only if $X_{ij}^{d^{-1}}=X_{id^{-1}jd^{-1}}\subseteq U$, hence $id^{-1}>jd^{-1}$. So $\row_\s(i)=\row_{\t^\lambda}(id^{-1})\geqslant \row_{\t^\lambda}(jd^{-1})=\row_\s(j)$. Now suppose $1\leqslant j<i\leqslant n$ and $ \row_\s(i)\geqslant\row_\s(j).$ So $X_{ij}\subseteq U$. If $\row_{\t^\lambda}(id^{-1})=\row_\s(i)> \row_\s(j)= \row_{\t^\lambda}(jd^{-1})$, we have $id^{-1}>jd^{-1}$ yielding $X_{ij}^{d^{-1}}=X_{id^{-1}jd^{-1}}\subseteq U$ and hence $X_{ij}\subseteq U^d$, as desired.
So let $\row_{\t^\lambda}(id^{-1})=\row_\s(i)= \row_\s(j)= \row_{\t^\lambda}(jd^{-1})$. Since $j<i$ by assumption and $\s\in \RStd(\lambda)$, $j$ is to the left of $i$ in $\s$ and hence $jd^{-1}$ is to the left of $id^{-1}$ in $\t^\lambda$ implying $jd^{-1}<id^{-1}$. Thus $X_{ij}^{d^{-1}}=X_{id^{-1}jd^{-1}}\subseteq U$ and hence $X_{ij}\subseteq U^d$ in this case as well.\end{proof}


 
 \begin{Defn}\label{4.4}
 Let $\lambda\vDash n, \s\in \RStd{\lambda}, d=\d(\s)$. Define the closed subsets of $\Phi^-$:
 \begin{enumerate}
 \item [1)] $P=P(\s)=\{(i,j)\in \Phi^-\,|\,\row_\s(i)\leqslant \row_\s(j)\}$
 \item [2)] $L=L(\s)=\{(i,j)\in \Phi^-\,|\,\row_\s(i)= \row_\s(j)\}\subseteq P$
 \item [3)] $I=I(\s)=\{(i,j)\in \Phi^-\,|\,\row_\s(i)<\row_\s(j)\}\subseteq P$
 \item [4)] $K=K(\s)=\{(i,j)\in \Phi^-\,|\,\row_\s(i)\geqslant \row_\s(j)\}\supseteq L$
 \item [5)] $J=J(\s)=\{(i,j)\in \Phi^-\,|\,\row_\s(i)>\row_\s(j)\}\subseteq K$
 \end{enumerate}
 \end{Defn}

The following Proüposition follows easily using \ref{3.8}, \ref{4.3} and (\ref{2.3}):

 \begin{Prop}\label{4.5}
 Keeping the notation introduced in \ref{4.3} and \ref{4.4} we have:
 \begin{enumerate}
 \item [1)] $P_\lambda^d\cap U=U_P, L_\lambda^d\cap U=U_L, U_\lambda^d\cap U=U_I, U^d\cap U=U_K, (U_\lambda^-)^d\cap U=U_J$.
 \item [2)] $U_I, U_L\leqslant U_P, U_I\trianglelefteq U_P, U_I\cap U_L=(1)$ and $U_P=U_IU_L$ (semidirect product).
 \item[3)] $U_J, U_L\leqslant U_K, U_J\trianglelefteq U_K, U_J\cap U_L=(1)$ and $U_K=U_JU_L$ (semidirect product).
 \item[4)] $U_P\cap U_K=U_L$.
 \item [5)] $U=U_PU_J, U_P\cap U_J=(1).$
 \item [6)] $U_L$ is conjugate by $d^{-1}$ to $L_\lambda\cap U$. In particular, $U_L$ is isomorphic to the direct product of the full unitriangular groups  $U_{\lambda_i}(q)$, $1\leqslant i\leqslant k,\lambda=(\lambda_1,\ldots, \lambda_k)\vDash n.$
 \end{enumerate}
  \end{Prop}

 Next we investigate the action of $U^d\cap U=U_K\leqslant U$ on the $\s$-component $M_\s$ of $\Res_{\text{\tiny$U$}}^{\text{\tiny$GL_n(q)$}} M(\lambda)$, $\lambda\vDash n, \s\in \RStd(\lambda), d=\d(\s)$.
\begin{Prop}\label{4.6}
Let $\lambda\vDash n, \s\in \RStd(\lambda), d=\d(\s)$. With the notation of \ref{4.5} the following holds:
$$
\Res_{U_K}^U M_\s \cong \Ind^{U_K}_{U_L} \C.
$$
\end{Prop}
\begin{proof}
This is Mackey decomposition again: First note that $U=U_PU_J\subseteq U_PU_K\subseteq U$ implies $U_PU_K=U$, i.e. there is only one $U_P U_K$ double coset in $U$ (with representative 1), hence 
\begin{eqnarray*}
\Res^U_{U_K} M_\s&=&\Res^U_{U_K} \Ind^ U_{U_P} \C\\
&\cong& \Ind ^{U_K}_{U_K\cap U_P} \Res^{U_P}_{U_K\cap U_P}\C\\
&=&\Ind ^{U_K}_{U_L} \C,
\end{eqnarray*}
by Mackey decomposition and \ref{4.5} part 4).
\end{proof}

Observe that  \ref{4.5} part 3) implies that $J, L\subseteq K$ satisfy the hypothesis \ref{2.8}. Thus we may apply the construction  of section \ref{sec2} to $U_K=U_J\rtimes U_L$. In particular, $\Ind_{U_L}^{U_K} \C$ is isomorphic to $\C \cE_J$, where $\cE_J$ is the lidempotent basis of the group algebra $\C V_J$ of the additive group $V_J=\{u-E\,|\,u\in U_J\}$. The monomial action of $U_K$ on $\cE_J$ is described in \ref{2.17}.


The set $\fX_\s^0:=\{du-d\,|\,u\in U_J\}\cong V_J$ as $\F_q$-vector space, since $d(u-E)=du-d$. For $A\in V_J$, $dA$ is again obtained from $A$ by reordering the rows. Since $M_\s$ has basis $\fX_\s$, the $\C$-basis $\cE_J$ of $\Res^U_{U_K} M_\s\cong \Ind^{U_K}_{U_L}\C$ may be expanded by linear combination of $\fX_\s$.


Applying left multiplication by $d$ to matrices in $V_J$, where rows of $dA, A\in V_J$ are relabeled  as in \ref{4.2-2} and extending this action by linearity to $\C (V_J,+)$, we turn the lidempotents $e_{_A}\in \cE_J, A\in V_J$ into lidempotents $e_{_{dA}}\in \C (\fX_\s^0,+)$, such that the $\C$-span of $\{e_{_{dA}}\,|\,A\in V_J\}$ is an $U_K$-module isomorphic to $\C \fX_\s$.

\begin{Theorem}\label{4.9}
Let $(i,j)\in K, \alpha\in \F_q, A\in V_J$. Then 
\[
e_{_{dA}}x_{ij}(\alpha)=\begin{cases}\theta(A_{ij}\alpha)e_{_{dB}}, & \text{ for } (i,j)\in J \\ e_{_{dB}}, & \text{ for } (i,j)\in L\end{cases}
\]
where $dB$ is obtained from $dA$ by adding $-\alpha$ times column $j$ to column $i$ of $dA$ and setting all entries of the resulting matrix to zero, which do not belong to $J$, if $(i,j)\in J$. If $(i,j)\in L$, we obtain $dB$ similarly by the combined truncated row and column operation. Moreover if $(i,j)\in L$ and belongs to highest (lowest) compartment in matrices in $\fX_\s$, then $x_{ij}(\alpha)$ acts by truncated column (row) operation alone. 
\end{Theorem}
\begin{proof}
Everything follows directly from  \ref{2.17} besides the last claim.
Let $(i,j)\in L$  such that  $(i,j)$ is a position in the highest $\s$-compartment of $dA$, that is, $\row_\s(i)=\row_\s(j)=1$. By construction for any position $(k,l)\in \Phi^-$ satisfying $\row_\s(k)=\row_\s(l)=1$ we have $(k,l)\notin J$ by Definition \ref{4.4}. In particular all entries in $dA$  in the highest compartment are zero. Thus only the truncated column operation adding $-\alpha$ times column $j$ to column $i$ in $dA$ and truncating can change $dA$ for the calculation of $e_{_{dA}} x_{ij}(\alpha)$. Now assume $\lambda=(\lambda_1, \ldots, \lambda_k)\vDash n$ and $(i,j)\in L$ with $\row_{\s}(i)=\row_\s (j)=k$. Then $(i,j)$ is a position in the lowest $\s$-compartment of $dA$. By \ref{xlambda} all entries in columns $i$ and $j$ in $dA$ are zero and hence truncated column operation adding  $-\alpha$ times column $j$ to column $i$ will not change $dA$. Thus the claim follows.\end{proof}

\begin{Prop}\label{4.10}
Let $\lambda=(\lambda_1, \dots, \lambda_k)\vDash n$ with $\lambda_k\neq 0$. Let $\s \in \RStd(\lambda)$ and set $J_k=\{(i,j)\in J\,|\,\row_\s(i)=k\}$. Then $J_k$ is closed in $\Phi^-$, and $U_{J_k}$ is an abelian normal subgroup of $U_J$ which acts on $e_{_{dA}}\in \C (\fX_\s^0,+)$ by a linear character, $d=\d(\s)\in \cD_\lambda$.
\end{Prop}
\begin{proof}
Since $J=\{(i,j)\in \Phi^-\,|\,\row_\s(i)>\row_\s(j)\}$  by definition, we obtain by direct calculation that $J_k$ is closed in $\Phi^-$ and $U_{J_k}$ is an abelian normal subgroup of $U_J$. Let $(i,j)\in J_k, \alpha\in \F_q$  and $A\in V_J$. Since row $i$ lies in the $k$-th compartment of $dA$, which is a lowest compartment, by \ref{xlambda} the $i$-th column of $dA$ is a zero column. Hence
 by \ref{4.9}, $e_{_{dA}}x_{ij}(\alpha)=\theta(A_{ij}\alpha)e_{_{dA}}$, which is a scalar action as desired. \end{proof}

\begin{Remark}\label{4.11}
Henceforth we identify  for fixed $d\in \cD_\lambda$ the spaces $V_J$ and $\fX_\s^0$ and think of matrices $dA\in \fX_s^0$ as elements of $(U_\lambda^-)^d\cap U$ with reordered rows keeping the original labeling of those as in \ref{4.2-2}. Thus for lidempotents $e_{_A}\in \cE_J$ we think as well $A$ to be a matrix in $V_J$ with reordered rows, sorting the rows of $A$ into consecutive compartments, each of those corresponding to a row in the tableau $\s=\t^\lambda d$.
\end{Remark}

\section{Supercharacters of $U$}\label{sec5}
The supercharacter theory of Andr\'{e} [\cite{andre1}] and Yan [\cite{yan}] is the special case of the construction in the previous section taking $\lambda=(1^n)\vdash n$ and $w=1\in \cD_\lambda$. Then $K=J=\Phi^-, \,L=\emptyset$ and the 1-cocycle $f: U \rightarrow V=V_{\Phi^-}=\Lie(U)$ is given by $f(u)=u-E\in V$. Note that this is a left 1-cocycle as well yielding a left action of $U$ on  $\C (V,+)$. Indeed  
$\C (V,+)$ is then a $\C U$-bimodule  isomorphic to the regular $\C U$-bimodule $_{\C U}\C U_{\C U}$.


We write $\cE=\cE_{\Phi^-}$, the set of lidempotents arising from $\C (V,+)$. To distinguish this special case notationally from other cases we denote now the lidempotent affording $\chi_{_{-A}}, A\in V$, by $[A]$ instead of $e_{_A}$.
For the convenience of the reader, we collect some well-known facts on the monomial $\C U$-bimodule $\C \cE$, where $\cE=\{[A]\,|\, A\in V\}$ is the $\C$-basis of $\C (V,+)$ consisting of lidempotents. For  details and proof we refer to [\cite{super,yan}].

\begin{Lemma}\label{truncated}
 Let $A\in V$,
  $1\leqslant j < i \leqslant n$ and  $\alpha\in \F_q$, then $$[A]x_{ij}(\alpha)=\theta(\alpha A_{ij})[A\ldotp x_{ij}(\alpha)],$$ where $A\ldotp x_{ij}(\alpha)$ is obtained from $A$ by adding $-\alpha$ times column $j$ to column $i$ (from left to right) and setting nonzero entries in the resulting matrix at position on or to the right of the diagonal to zero. This is called {\bf   truncated column operation}, (comp. \ref{2.199}). Similarly the left operation of $x_{ij}(\alpha)$ on the lidempotent basis $\{[A]\,|\,A\in V\}$ of $\C V$ can be described by a truncated row operation from down up, the coefficient in $\C$ being again $\theta(\alpha A_{ij}) $. \hfill $\square$
\end{Lemma}

\begin{Defn}\label{def of verge}
A subset $\p=\{(i_1, j_1), \ldots, (i_k, j_k)\}\subseteq \Phi^-$ is called a {\bf main condition set} if 
$\ppi=\{i_1, \ldots, i_k\}$ and $\pj=\{j_1, \ldots, j_k\}$ are sets of $k$ many pairwise different indices in $\{1, \ldots, n\}$. So $\p$ picks from each row and each column of $n\times n$-matrices at most one position. We call $[A]\in \cE$ {\bf verge} if $\supp(A)=\p\subseteq \Phi^-$ is a main condition set and call then the elements of $\p$ main conditions.
\end{Defn}

It is easy to see that verge lidempotents correspond to the ``basic characters'' defined by Andr\'{e}. Indeed it can be shown that each $U$-$U$-biorbit of $\cE$ contains exactly one verge $[A]\in \cE$ and  all right orbits contained in the biorbit generated by the verge $[A]\in \cE$ afford identical characters of $U$. Those are precisely the supercharacters. Distinct biorbits afford orthogonal characters and hence each irreducible character of $U$ is irreducible constituent of precisely one supercharacter, (comp. [\cite{mindim}, 2.13]).


We denote for $[A]\in \cE, A\in V$, the right orbit containing $[A]$ by $\cO_A$.

\begin{Defn}\label{5.2}
 Let $1\leqslant j < i \leqslant n$. The {\bf hook arm} $h_{ij}^a$ centred at $(i,j)$ consists of all positions $(i,k)\in \Phi^-$ strictly to the right of $(i,j)$, thus $h_{ij}^a=\{(i,k)\,|\,j<k<i\}$, and the {\bf hook leg} $h_{ij}^l$ is the set of positions $(l,j)\in \Phi^-$ strictly above $(i,j)$, thus $h_{ij}^l=\{(l,j)\,|\,j<l<i\}$. Finally the {\bf hook} $h_{ij}$ centred at $(i,j)$ is defined to be $h_{ij}=h_{ij}^a\cup h_{ij}^l\cup \{(i,j)\}$. Let $\p\subseteq \Phi^-$ be a main condition set. The hooks centred at positions in $\p$ are called {\bf main hooks.}
\end{Defn}

\begin{Theorem}[][\cite{mindim}]\label{pstab}
 Let $[A]\in \cE$ be a verge lidempotent with $ \supp(A)=\p_{_A}=\{(i_1,j_1),\ldots,(i_k,j_k)\}$ $\subseteq \Phi^-$. Then the right projective stabilizer of $[A]$ in $U$, that is the set
$\{u\in U\,|\, [A]u=\lambda [A], \text{ for some } \lambda \in \C^*\}$,
denoted by $\Pstab_{\mbox{\tiny $U$}}[A]$, is a pattern subgroup $U_{\cal R}$ with 
\[
{\cal R}=\{(r,s)\in \Phi^-\,|\,s\notin \{j_1,\ldots,j_k\}\}\cup\{(r,j_\nu)\,|\,\nu=1,\ldots,k, i_\nu\leqslant r \leqslant n\}.
\] 
Thus $\p_{_A}\subseteq {\cal R}$ and ${\cal R}^\circ ={\cal R}\setminus \p_{_A}$ is closed. Moreover $U_{{\cal R}^\circ }$ acts trivially on $[A]$, $U_{{\cal R}^\circ }\trianglelefteq U_R$ and $U_{\cal R}/U_{{\cal R}^\circ }\cong X_{i_1 j_1}\times \cdots\times X_{i_k j_k}$
acting on $[A]$ by the linear character $\theta_A=\theta_1\times \theta_2\times \cdots \theta_k$, where
$\theta_{\nu}: X_{i_\nu j_\nu}\rightarrow \C^*$ sends $x_{i_\nu j_\nu}(\alpha)$ to $ \theta( \alpha A_{i_\nu j_\nu})\in \C^*$ for $\alpha\in \F_q,\, \nu=1,\ldots,k$. \hfill $\square$
\end{Theorem}

Thus ${\cal R}$ consists of all positions in $\Phi^-$ in zero columns of $A$ together with all positions on and below the positions in $\p_{_A}$.

\begin{Defn}\label{5.5}
Let $[A]\in \cE$ be a verge with $ \supp(A)=\p_{_A}\subseteq \Phi^-$. $U_\cR=\Pstab_{\mbox{\tiny $U$}}[A]$ is defined as above. We define
$\hat \cR$ to be $\cR$ combined  with all positions on hook legs, such that the corresponding subgroups change in  $[A]$ only the values at a  hook intersection acting from the right. We illustrate as follows:
\end{Defn}

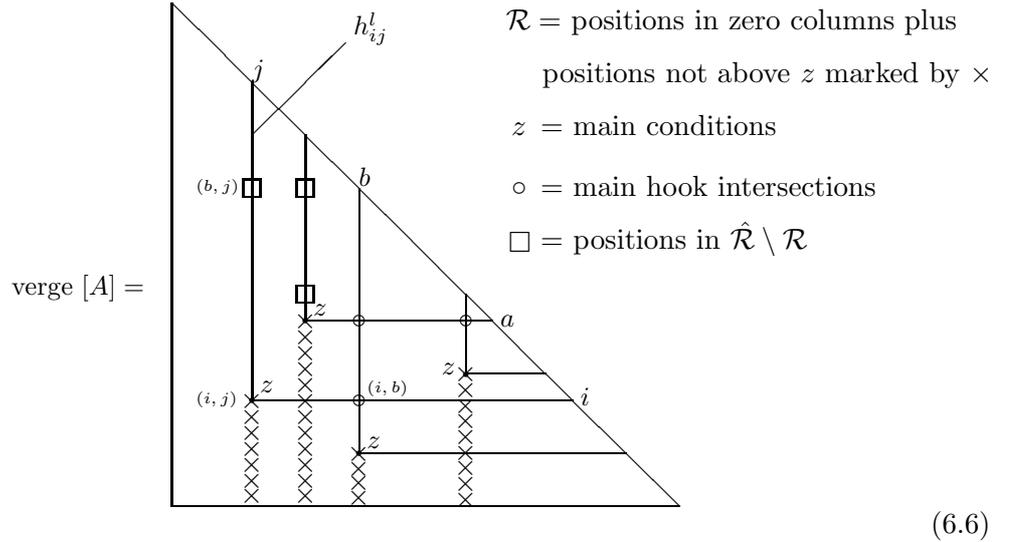
\begin{equation}
\begin{picture}(200,200)
\put(20,10){\line(0,1){190}}
\put(20,10){\line(1,0){190}}
\put(20,200){\line(1,-1){190}}

\put(50,50){\line(0,1){120}}
\put(50,50){\line(1,0){120}}
\put(50,50){\circle*{2}}
\put(53,53){\small$z$}

\put(70,80){\line(0,1){70}}
\put(70,80){\line(1,0){70}}
\put(70,80){\circle*{2}}
\put(73,82){\small$z$}

\put(130,60){\line(0,1){30}}
\put(130,60){\line(1,0){30}}
\put(130,60){\circle*{2}}
\put(121,60){\small$z$}

\put(90,30){\line(0,1){100}}
\put(90,30){\line(1,0){100}}
\put(90,30){\circle*{2}}
\put(93,32){\small$z$}

\put(90,50){\circle{4}}
\put(90,80){\circle{4}}
\put(130,80){\circle{4}}

\put(90,131){\small$b$}
\put(50,172){\small$j$}
\put(143,78){\small$a$}
\put(173,48){\small$i$}

\put(46.5,130){\framebox{}}

\put(66.5,130){\framebox{}}

\put(66.5,90){\framebox{}}

\put(50, 150){\line(1,1){35}}
\put(88, 188){\small$h_{ij}^l$}

\put(29,129){\tiny$(b,j)$}
\put(29,49){\tiny$(i,j)$}
\put(93,53){\tiny$(i,b)$}

\multiput(45.6,47)(0,-6){7}{$\times$}

\multiput(65.6,77)(0,-6){12}{$\times$}

\multiput(125.6,57)(0,-5.89){9}{$\times$}

\multiput(85.6,27)(0,-5.5){4}{$\times$}

\put(-40,90){\small verge $[A]=$}

\put(150,130){\circle{4}}
\put(158,127){= main hook intersections}
\put(147,150){$z$}
\put(158,150){= main conditions}

\put(146,106){$\square$}
\put(158,107){= positions in $\hat\cR\setminus \cR$}

\put(145,190){$\cR=$ positions in zero columns plus}
\put(155,170){ positions not above $z$ marked by $\times$}
\end{picture}
\end{equation}
For example $X_{bj}$ acting on $[A]$ will only change the entry at the main hook intersection $(i,b)$.


It was shown in [\cite{mindim}, 4.5], that $\hat \cR$ is a closed subset of $\Phi^-$ with $U_{\cR^o}, U_\cR\trianglelefteq U_{\hat \cR}$ and that $U_{\cal R}/U_{{\cal R}^\circ }\cong X_{i_1 j_1}\times \cdots\times X_{i_k j_k}$ is a central subgroup of $U_{\hat \cR}/ U_{\cR^o}$.

\begin{Defn}\label{hookconnected}
Let $$\p=\{(i_1, j_1),\ldots,(i_k,j_k)\}\subseteq \Phi^-$$ being a main condition set. We call $\p$ {\bf completely hook disconnected} if $\ppi=\{i_1, \ldots, i_k\}$ and $\pj=\{j_1, \ldots, j_k\}$ are disjoint. Thus, in this case, $\{i_1, \ldots, i_k, j_1, \ldots, j_k\}$  is a subset  of $\{1, \ldots,n\}$ of $2k$ many pairwise different indices.
\end{Defn}

 Thus if $[A]\in \cE$ is a verge with completely disconnected condition set $\p=\supp(A)$, no main hooks of $\p$ meet at the diagonal, that is $(a,i), (i,j) \in \p$ does not occur:

For our main application for section \ref{sec7}, this condition is automatically satisfied and is the special case of hook disconnected main condition sets defined in [\cite{mindim}, 5.1]. We state here the main result of [\cite{mindim}] for those, as far as they are needed here.

\begin{Results}[][\cite{mindim}]\label{5.8}
Let $\p\subseteq \Phi^-$ be a (completely) hook disconnected condition set and $[A]\in \cE$ be a verge with $\supp(A)=\p$. Then the following holds:
\begin{enumerate}
\item [1)] $ \hat  \cR ^{-} =\hat \cR\setminus \p$ is  closed in $\Phi^-$ with $U_{ \hat  \cR ^{-}}\trianglelefteq U_{\hat \cR}$. 
\item [2)] $\End_{\C U} (\C \cO_A)\cong \C( U_{ \hat  \cR}/U_{\cR})\cong \C(U_{\hat \cR^-}/U_{\cR^o})=\C H$.
\item [3)]  $U_{\hat \cR}/U_{\cR^o}\cong H\times X_{i_1j_1}\times \cdots \times X_{i_kj_k}$.
 If $S$ is an irreducible $\C H$-module, extending the action of $H$ on $S$ by the linear character $\theta_A$ defined in \ref{pstab} and letting $U_{\cR^o}$ act trivially yields an irreducible $\C U_{\hat \cR}$-module $\hat S$ such that $[A] \hat S \cong \hat S$ and $[A] \hat S \C U \cong \Ind^U_{U_{\hat \cR}}\hat S$ is an irreducible constituent of $\C \cO_A$.
\item [4)] $S\mapsto \hat S \mapsto [A]\hat S \C U$ is a multiplicity preserving bijection between the irreducible constituents of the group algebra $\C H$ and those of the $\C U$-module $\C \cO_A$. \hfill $\square$
\end{enumerate}
\end{Results}

\begin{Remark}\label{5.9}
So, in particular, choosing $S$ to be the trivial $\C H$-module, $[A]\hat S \C U\leqslant \C \cO_A$ is a unique irreducible constituent (of multiplicity one) of $\C \cO_A$, isomorphic to $\Ind^U_{U_{\hat \cR}} \C \epsilon_{\!_A}$, where $\epsilon_{\!_A}$ is the primitive central idempotent of $\C U_{\hat \cR}$ affording the linear character $\theta_A$ on $U_{\cR}/U_{\cR^o}$ and the trivial character on $H$. In particular, $\Stab_{\mbox{\tiny{$U$}}}(\epsilon_{\!_A})=U_{\hat \cR^-}\trianglelefteq U_{\hat \cR}=\Pstab_{\mbox{\tiny $U$}}(\epsilon_{\!_A})$.
\end{Remark}

\section{Two part compositions}\label{sec6}
In this section we apply the general method of section \ref{sec4} to the special case that $\lambda=(n-m,m), 1\leqslant m\leqslant n-1$ is a composition of $n$ into two parts. Thus $P_\lambda\leqslant GL_n(q)$ is a maximal parabolic subgroup and all maximal parabolic subgroups of $GL_n(q)$ are conjugate to some $P_\lambda, \lambda=(n-m,m)\vDash n$.  Moreover, if $\lambda=(n-m,m)\vDash n$ then $\mu=(m, n-m)\vDash n$ too and it is a well known fact, that $\Ind_{P_\lambda}^{GL_n(q)} \C \cong \Ind^{GL_n(q)}_{P_\mu} \C$. Thus, in the following, we may always assume that $\lambda=(n-m,m)$ is indeed a partition of $n$, that is $m\leqslant n-m$.
Note that the set $\cF(\lambda)$ of $\lambda$-flags is the set of $m$-dimensional $\F_q$-subspaces of $\F_q^n$. By \ref{3.11} we may identify $\cF(\lambda)$ by matrices in $\fX_\lambda=\{du\,|\,d\in \cD_\lambda, u\in (U_\lambda^-)^d\cap U\}$ where each matrix in $\fX_\lambda$ is divided into two compartments, the upper compartment of $A\in \fX_\lambda$ consisting of the first $(n-m)$ and the lower one of the last $m$ rows of $A$. For $\s\in \RStd(\lambda)$, we denote the second row of $\s$ by $\ss$, thus $\ss=(i_1, \ldots, i_m)$ with $1\leqslant i_1< i_2<\cdots<i_m\leqslant n$. Obviously $\ss$ determines $\s$ uniquely the first row of $\s$ consisting of all numbers $i\in \{1, \ldots, n\}$ with $i\notin \ss$ ordered from left to right by the natural ordering of $\N$. Note that by our convention introduced in section \ref{sec4}, the rows in the lower compartment of $A\in \fX_\s^0$ are  labelled by $i_1, \ldots, i_m$, (comp. \ref{4.11}). Fix $\s\in \RStd(\lambda), d=\d(\s)$ and recall Definition \ref{4.4}. In particular:
\begin{eqnarray}
L&=&L(\s)=\{(i,j)\in \Phi^-\,|\,i,j\in \ss \text{ or } i,j\notin \ss\}\nonumber\\
J&=&J(\s)=\{(i,j)\in \Phi^-\,|\, i\in \ss, j\notin \ss\}\nonumber\\
K&=&K(\s)=L\cup J.\label{6.1}
\end{eqnarray}
Let $e_{_A}\in \cE_J$, then $A+E\in U_J$. Then all entries in the first compartment of $A$ are zero, and hence $e_{_A}$ is entirely determined by the lower compartment. Thus, in illustrations, we may omit the first compartment.

\begin{Remark}\label{action}
By \ref{4.10} $U_J$ is abelian and acts on the lidempotents in $\cE_J$ by linear characters. Moreover $L$ splits into $L_1=\{(i,j)\in \Phi^-\,|\, i,j\notin \ss\}$ and $L_2=\{(i,j)\in \Phi^-\,|\, i,j\in \ss\}$, the positions in $L_1$ belonging to the upper and in $L_2$ to the lower compartment in matrices in $\fX_\s$. By \ref{4.9}, $U_{L}=U_{L_1}\times U_{L_2}$, $U_{L_1}\cong U^-_{n-m}(q), U_{L_2}\cong U^-_{m}(q)$, where $U_{L_1}$ acts by truncated  column and $U_{L_2}$ by truncated row operations on $\cE_J$ permuting $\cE_J$. For $e_{_A}\in \cE_J$ we denote the $U_K$-orbit in $\cE_J$ containing $e_{_A}$ by $\cO_{_A}^J$ as in \ref{orbit}.
\end{Remark}

Throughout this section we shall use the setting and notation introduced above without further notice.
The permutation module $\Res_{U}^{GL_n(q)} M(\lambda)$ has been investigated in [\cite{brandt, dj1}] and [\cite{guo}] using the basis  of $M(\lambda)$ consisting of $\lambda$-flags. The basis $\fX_\lambda$ was constructed there. Comparing Proposition \ref{4.9} with section 2 of [\cite{guo}] shows that indeed our construction here contains the exposition in [\cite{guo}] as a special case.


For the convenience of the reader we summarize the relevant results on $M_\s$ shown in [\cite{guo}].

\begin{Results}\label{results}
\begin{itemize}
\item [(1)] [\cite{guo}, 2.5.6]  Each $U_K$-orbit $\cO^J$ of $\cE_J$ contains a unique lidempotent $e_{_A}, A\in V_J$, such that in each  row and in each column of $A$ there is at most one non-zero entry. Similarly as in \ref{def of verge} we call such  lidempotent  {\bf verge (lidempotent)} of $\cE_J$ and define $ \main(\cO^J)=\main(e_{_A})=\supp(A)$.  Note, that then $\p=\p_{_A}=\supp ( A)$ is a main condition set in $\Phi^-$ as defined in \ref{def of verge}. In particular, (putting the rows of $A$ again in the natural order) $[A]\in \cE$ is a verge in the lidempotent basis of $\C (V_{\Phi^-},+)$.

\item [(2)]  Let $e_{_A}\in \cE_J$ be a verge and $\p=\p_{_A}= \main(e_{_A})$ its main condition set. Recall from \ref{def of verge} the definition of $\ppi=\{i\in \{1, \ldots,n\}\,|\, \exists\, 1\leqslant j<i: (i,j)\in \p\}$ and $\pj=\{j\in \{1, \ldots,n\}\,|\, \exists\,  j<i\leqslant n: (i,j)\in \p\}$. Let $\mu=(n-k,k)\vDash n, \t\in \RStd(\mu)$. We say  $\p$ {\bf fits} the $\t$-component $M_\t$ of  $\Res^{GL_n(q)}_U M(\mu)$, if $\ppi\subseteq \tt$ and $\pj\cap \tt=\emptyset$.
\item [(3)] [\cite{guo}, 2.5.10] Let $e_{_A}\in \cE_J$ be a verge with main condition set  $\p=\p_{_A}\subseteq \Phi^-$. 
\begin{itemize}
\item [\romannumeral1)]  Define $L_1^0=\{(i,j)\in L_1\,|\,  j\notin \pj\text{ or }\exists\, (b,j)\in \p \text{ with } b<i\}$. Thus $L_1^0$ consists of all positions $(i,j)$ in $L_1$ in the upper compartment of $dA$, $d=\d(\s)\in \cD_\lambda$, where either column $j$ is a zero column (if and only if $j\notin \pj$) or there is a main condition $(b,j)$ in column $j\in \pj$, above position $(i,j)$ in $A$ (in the natural order of rows of $A$). 


Note that in both cases $(i,j)\in \cR^o$ and hence $X_{ij}\in \Stab_{\mbox{\tiny$U$}} [A]$ by \ref{pstab}. We illustrate the case $j\in \pj$. Let $ \alpha\in \F_q$:
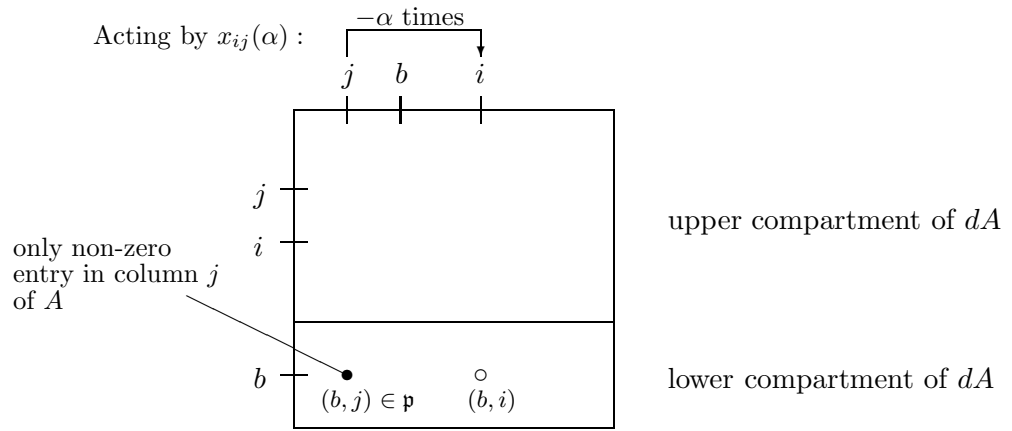
\begin{center}
\begin{picture}(220,170)
\put(30,0){\line(0,1){120}}
\put(30,0){\line(1,0){120}}
\put(150,0){\line(0,1){120}}
\put(150,120){\line(-1,0){120}}
\put(30,40){\line(1,0){120}}
\put(25,20){\line(1,0){10}}
\put(15,15){$b$}
\put(25,70){\line(1,0){10}}
\put(15,65){$i$}
\put(25,90){\line(1,0){10}}
\put(15,85){$j$}
\put(50,20){\circle*{4}}
\put(40,8){\footnotesize$(b,j)\in \p$}

\put(100,20){\circle{4}}
\put(95,8){\footnotesize$(b,i)$}

\put(30,0){\line(1,0){120}}
\put(50,115){\line(0,1){10}}
\put(48,130){$j$}
\put(70,115){\line(0,1){10}}
\put(68,130){$b$}
\put(100,115){\line(0,1){10}}
\put(98,130){$i$}

\put(50,140){\line(0,1){10}}
\put(50,150){\line(1,0){50}}
\put(100,150){\vector(0,-1){10}}
\put(53,152){\small$-\alpha$ times}
\put(170,75){upper compartment of $dA$}
\put(170,15){lower compartment of $dA$}

\put(-45,145){\small Acting  by $x_{ij}(\alpha):$}

\put(50,20){\line(-2,1){60}}

\put(-75,65){\small only non-zero}
\put(-75,55){\small entry in column $j$}
\put(-75,45){\small  of $A$}
\end{picture}
\end{center}


If $(b,j)\in \p$ with $b<i$, then $(b,i)\notin \Phi^-$ and hence $(b,i)\notin J$. Now $e_{_A}x_{ij}(\alpha)=e_{_B}$, where $B$ is obtained from $A$ by adding $-\alpha$ times column $j$ to column $i$ and projecting the resulting matrix into $V_J$. Since $(b, i)\notin J$ we conclude that $e_{_A}x_{ij}(\alpha)=e_{_A}$. Indeed $L_1^0$ is closed in $\Phi^-$ and $U_{L_1^0}=\Stab_{\mbox{\tiny $U_{L_1} $}}(e_{_A})$.

\item[\romannumeral2)] Define $L_2^0=\{(i,j)\in L_2\,|\  i\notin \ppi\text{ or }\exists\, (i,v)\in \p \text{ with } v>j,\}$. Thus $L_2^0$ consists of positions $(i,j)$ in $L_2$ in the lower compartment of $dA$, where either contained in a zero row  of $A$ (if and only if $i\notin \ppi$) or the main condition in row $i$ is to the right of $(i,j)$. Again we illustrate the situation in the second case $i\in \ppi$ (ommiting from $dA$ the upper compartment): Let $ \alpha \in \F_q$:
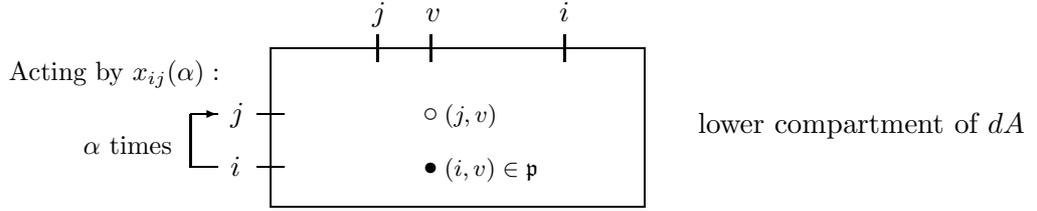
\begin{center}
\begin{picture}(250, 95)
\put(30,10){\line(1,0){140}}
\put(30,70){\line(1,0){140}}
\put(30,10){\line(0,1){60}}
\put(170,70){\line(0,-1){60}}
\put(70,65){\line(0,1){10}}
\put(68,80) {$j$}
\put(90,65){\line(0,1){10}}
\put(88,80) {$v$}
\put(140,65){\line(0,1){10}}
\put(138,80) {$i$}
\put(25,45){\line(1,0){10}}
\put(15,42) {$j$}
\put(25,25){\line(1,0){10}}
\put(15,22) {$i$}
\put(90,25){\circle*{4}}
\put(95,22){\footnotesize$(i,v)\in \p$}
\put(90,45){\circle{4}}
\put(95,42){\footnotesize$(j,v)$}
\put(190,38){lower compartment of $dA$}
\put(10,25){\line(-1,0){10}}
\put(0,25){\line(0,1){20}}
\put(0,45){\vector(1,0){10}}
\put(-40,30){\small$\alpha$ times}
\put(-68,58){\small Acting  by $x_{ij}(\alpha):$}
\end{picture}
\end{center}
Again, since $j<v$ and $(j,v)$ is the only position being changed in $dA$, when adding $\alpha$ times row $i$ to row $j$. But this entry at position $(j,v)$ is set back to zero by truncation. Thus $e_{_A}x_{ij}(\alpha)=e_{_A}$, indeed $L_2^0\subseteq\Phi^-$ is closed and $U_{L_2^0}=\Stab_{\mbox{\tiny $U_{L_2} $}}(e_{_A})$.
\item  [\romannumeral3)] The stabilizer of $e_{_A}$ in $U_K$ is not a pattern subgroup in general. To see this consider the following situation: 
Let $(s,i), (t,j)\in \p$ with $1\leqslant j<i<t<s\leqslant n$ and $\alpha\in \F_q$:

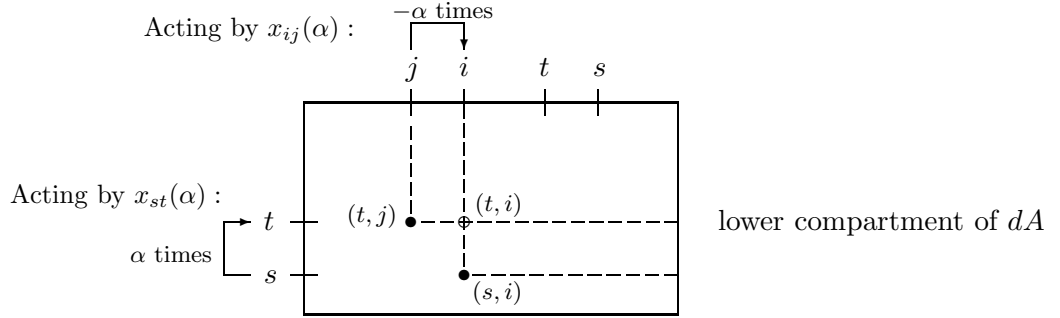
\begin{center}
\begin{picture}(225,130)
\put(-80,52){\small Acting by $x_{st}(\alpha):$}
\put(30,10){\line(1,0){140}}
\put(30,90){\line(1,0){140}}
\put(30,10){\line(0,1){80}}
\put(170,90){\line(0,-1){80}}
\put(70,85){\line(0,1){10}}
\put(68,100) {$j$}
\put(90,85){\line(0,1){10}}
\put(88,100) {$i$}
\put(120,85){\line(0,1){10}}
\put(118,100) {$t$}
\put(140,85){\line(0,1){10}}
\put(138,100) {$s$}
\put(25,45){\line(1,0){10}}
\put(15,42) {$t$}
\put(25,25){\line(1,0){10}}
\put(15,22) {$s$}
\put(90,25){\circle*{4}}
\put(93,16){\footnotesize$(s,i)$}
\put(90,45){\circle{4}}
\put(94,49){\footnotesize$(t,i)$}
\put(185,42){lower compartment of $dA$}
\put(10,25){\line(-1,0){10}}
\put(0,25){\line(0,1){20}}
\put(0,45){\vector(1,0){10}}
\put(-35,30){\footnotesize$\alpha$ times}

\put(70,45){\circle*{4}}
\put(46,45){\footnotesize$(t,j)$}

\put(70,110){\line(0,1){10}}
\put(70,120){\line(1,0){20}}
\put(90,120){\vector(0,-1){10}}
\put(63,122){\footnotesize$-\alpha$ times}

\put(-30,115){\small Acting by $x_{ij}(\alpha):$}

\multiput(74,45)(7,0){14}{\line(1, 0){5}}
\multiput(70,48)(0,7){5}{\line(0, 1){5}}

\multiput(93,25)(7.2,0){11}{\line(1, 0){5}}
\multiput(90,28)(0,7){9}{\line(0, 1){5}}

\end{picture}
\end{center}

Then $(t,i)\in J$ is a main hook intersection. Note that $i,j\notin \ss$, $s, t\in \ss$ hence $(i,j)\in L_1, (s, t)\in L_2$ are in the upper respectively in the lower compartment. Acting by $x_{ij}(\alpha)$ ($\alpha\in \F_q$) on $e_{_A}$ adds $-\alpha$ times column $j$ to column $i$ and hence inserts $-\alpha A_{tj}$ into position $(t,i)$. Acting by $x_{st}(\beta)$ ($\beta\in \F_q$) adds $\alpha$ times row $s$ to row $t$ and hence inserts $\beta A_{si}$ into position $(t,i)$. In both cases, the entry at position $(t, i)$ is the only one which is changed. Note that by assumption $A_{tj}\neq 0 \neq A_{si}$ and hence choosing $\beta=\alpha \frac{A_{tj}}{A_{si}}\in \F_q$ we have $e_{_A} x_{ij}(\alpha)x_{st}(\beta)=e_{_A}$, that is $x_{ij}(\alpha)x_{st}(\beta)\in \Stab_{\mbox{\tiny $U_{K} $}}(e_{_A})$.
Note that inspecting \ref{pstab} we see that $(i,j)\in \hat \cR\setminus \cR$ and hence $X_{ij}\leqslant U_{\hat \cR}$.
\end{itemize}

\item[(4)] [\cite{guo}, 2.5.14] Let $ {\cO^J}\subseteq\cE_J$ be an $U_K$-orbit. Then   $\C \cO^J$ is an irreducible $\C U_K$-module.
\item [(5)] Recall from \ref{4.6} and \ref{2.11} that $\Res^U_{U_K} M_\s\cong \C \cE_J$, where $M_\s=\Ind^{U}_{U_P} \C \overline{P_\lambda}\,d$ by \ref{mackey}, $d=\d(\s)\in \cD_\lambda, U_P=P_\lambda^d\cap U$ (by \ref{4.5}) and $\overline{P_\lambda}=\sum_{h\in P_\lambda} h$. An isomorphism from $\C \cE_J$ to $\Res^{U}_{U_K} M_\s$ can be described explicitly by $f^*: \C (V_J,+)\rightarrow \C U_K: \tau \mapsto \tau \circ f$ for $\tau \in \C^{V_J}\cong \C (V_J,+)$, see \ref{1.8}. Here $f: U_K\rightarrow V_J$ is the 1-cocycle of \ref{2.10}. Recall that $f|_{U_J}:U_J\rightarrow V_J$ is given as $f(u)=u-E$ by the proof of \ref{2.10}. Setting $\hat  U_L=q^{-|L|}\sum_{x\in U_L}x$, we have now $f^*(A)=\hat U_L(A+E)\in \C U_K$, since   $U_L=\ker f$ by \ref{1.9}. Hence 
\[
f^*(e_{_A})=f^*(q^{-|J|}\sum\nolimits_{B\in V_J}\overline {\chi_{_A}(B)}B)=q^{-|J|}\hat U_L \sum\nolimits_{u\in U_J}\overline {\chi_{_A}(u-E)}u
\]
Observing that $U_L\subseteq U_P\subseteq P_\lambda^d, M_\s=\Ind^U_{U_P} \C \overline {P_\lambda} d=\overline {P_\lambda} d\,\C U$, we see that $\overline {P_\lambda} d=\overline {P_\lambda} d\hat U_P=\overline {P_\lambda} d\hat U_L$, where $\hat U_P=q^{-|P|}\sum_{x\in U_P}x$. Therefore we may identify $f^*(e_{_A})$ with $\sum\limits_{u\in U_J}\overline {P_\lambda} d\,\overline {\chi_{_A}(u-E)}u$   $\in \overline {P_\lambda} d\, \C U\cong M_\s$.
Let $ {\cO^J}\subseteq\cE_J$ be an $U_K$-orbit, $\tilde \cO^J\subseteq \overline{P_\lambda}d \C U$
its image under the identifications, then it was shown in [\cite{guo}, 2.6.2], that $\C \tilde \cO^J\subseteq \overline{P_\lambda}d \C U $ is $U$-invariant. Since $\C \tilde \cO^J\cong \C \cO^J$ is an irreducible $\C U_K$-module, we conclude that the $U_K$-action on $\C \cO^J$ can be extended to $U$ yielding the irreducible $\C U$-module $\C \cO^J\cong \C \tilde \cO^J$.
\hfill$\square$\end{itemize}
\end{Results}

\begin{Remark}
For a given main condition $\p\subseteq \Phi^-$, it might fit many irreducible orbit modules for different components and even for different 2-part partitions. For example: Let  $0\neq \alpha\in \F_q$.
$$\begin{matrix}
\begin{tabular}{p{2mm}c}
                                  &\begin{tabular}{*{6}{p{2mm}}}\emph 1   & \emph  2 &\emph3  &\emph4 &\emph5&\emph6\end{tabular}\\
\begin{tabular}{c}\emph3 \\\emph5 \\\emph6 \end{tabular} &\begin{tabular}{|*{6}{p{2mm}|}}\hline
                                     {$\ast$}&  {$\square$}  & 1& & & \\\hline
                                    {$\alpha$}&   {$\ast$}  &0 &  {$\ast$}&1 & \\\hline &&0&&0&1\\\hline
                                   \end{tabular}
\end{tabular}

& &

\begin{matrix}
\lambda=(3,3),&
\t=\begin{tabular}{|c|c|c|}\hline 1&2&4\\\hline 3&5&6\\\hline\end{tabular}
\end{matrix}
\end{matrix}$$

$$\begin{matrix}
\begin{tabular}{p{2mm}c}
                                  &\begin{tabular}{*{6}{p{2mm}}}\emph 1   &   \emph2 &\emph3  &\emph4 &\emph5&\emph6\end{tabular}\\
\begin{tabular}{c}\emph2 \\\emph5 \\\emph6 \end{tabular} &\begin{tabular}{|*{6}{p{2mm}|}}\hline
                                     {$\ast$}&  1  & & & & \\\hline
                                   {$\alpha$}& 0&  {$\ast$}   &   {$\ast$}&1 & \\\hline &0&&&0&1\\\hline
                                   \end{tabular}
\end{tabular}
& &
\begin{matrix}
\lambda=(3,3),&\t=\begin{tabular}{|c|c|c|}\hline 1&3&4\\\hline 2&5&6\\\hline\end{tabular}
\end{matrix}
\end{matrix}$$

$$\quad\,\,\,\begin{matrix}
\begin{tabular}{p{2mm}c}
                                  &\begin{tabular}{*{6}{p{2mm}}} \emph1   &  \emph 2 &\emph3  &\emph4 &\emph5&\emph6\end{tabular}\\
\begin{tabular}{c}\emph2 \\\emph5 \end{tabular}&\begin{tabular}{|*{6}{p{2mm}|}}\hline
                                      {$\ast$}&  1  & & & & \\\hline
                                {$\alpha$}& 0&  {$\ast$}   &  {$\ast$}&1 & \\\hline
                                   \end{tabular}
\end{tabular}& &
\begin{matrix}
\lambda=(4,2),& \t=\begin{tabular}{|c|c|c|c|}\hline 1&3&4&6\\\hline 2&5\\\cline{1-2}\end{tabular}
\end{matrix}
\end{matrix}$$
The three orbits above have the same  main condition set $\p=(5,1)$ with the filling $\alpha$ and the same dimension $ q^3$. 
\end{Remark}

In \ref{results} part (4) and (5) we have seen that the $\C$-spaces spanned by orbits $\cO^J\subseteq \cE_J$ are irreducible $\C U$-modules. The next theorem states, that these $\C U$-modules depend not really on $\cO^J\subseteq \cE_J$ or even on the 2-part partition $\lambda=(n-m,m)$, but only on the main condition set $\p=\main (\cO^J)$ and the non-zero values of $A$ at positions in $\p$ for the unique verge lidempotent $e_{_A}\in \cO^J$.

\begin{Theorem}[][\cite{guo}, 3.1.30] \label{irr2} 
  Let $\lambda, \mu$ be 2-part partitions of $n$ and let $\s\in \RStd(\lambda), \t\in \RStd(\mu)$. Let $e_{_A}\in \cE_{J(\s)}$ and $e_{_B}\in \cE_{J(\t)}$ be verges with $A=B$ in $ M_n(q)$. Thus 
$\main( \cO_A^{J(\s)})=\supp(A)=\p=\supp(B)=\main( \cO_B^{J(\t)})$ and $A_{ij}=B_{ij}$ for all $(i,j)\in \p$. Then $\C \cO_A^{J(\s)}\cong \C \cO_B^{J(\t)}$ as $\C U$-modules.
\end{Theorem}

As a consequence of the main result \ref{7.16} of this paper we obtain that the converse of Theorem \ref{irr2} holds as well, that is $\C \cO_A^{J(\s)}\cong \C \cO_B^{J(\t)}$ if and only if $A=
B$ in $M_n(q)$.

\begin{Remark}\label{6.6}
Let $\lambda=(n-m,m)$ be a partition of $n$, $\s\in \RStd(\lambda)$ and $\cO^J\subseteq \cE_J$ an $U_K$-orbit, $J=J(\s)$ as above. Let $\p=\{(i_1, j_1), \ldots, (i_k ,j_k)\}=\main (\cO^J)$ the set of main conditions for the unique verge $e_{_A}$ ($A\in V_J$) in $\cO^J=\cO_A^J$. Since $\p\subseteq J=\{(i,j)\in \Phi^-\,|\, i\in \ss, j\notin \ss\}$ we conclude that $\ppi\cap \pj=\{i_1, \ldots, i_k\}\cap\{j_1, \ldots,j_k\}=\emptyset$. Thus  $\p\subseteq \Phi^-$ is a completely hook disconnected main condition set. Moreover $k\leqslant m$. Let $\ss\setminus\ppi=\{i_{k+1}, \ldots, i_m\}$, then $\mu=(n-k,k)$ is as well a partition of $n$ and $\t\in \RStd(\lambda)$ with $\tt=(i_1,\ldots, i_k)$ (assuming as we always do, that $1\leqslant i_1<i_2<\cdots<i_k\leqslant n$). It is obvious that  $\p$ fits $\t$ as well and $A\in V_{J(\t)}$ as well. By \ref{irr2} $\C \cO_A^{J(\s)}\cong \C \cO_A^{J(\t)}$. This indeed works for arbitrary completely hook disconnected main condition sets $\p=\{(i_1, j_1), \ldots, (i_k, j_k)\}\subseteq \Phi^-$. Since then $\ppi\cap \pj=\emptyset$, we always have $2k\leqslant n$ and hence $(n-k,k)$ is a partition of $n$, and there is a unique row standard $\lambda$-tableau $ \t $ with $\tt=\ppi=(i_1, \ldots, i_k)$. We remark that $A$ is in fact standard, that is increasing in the columns as well. This is not hard to see, but is not needed in this paper. If $|\p|=k$ and $\lambda=(n-k,k)$, each row of the lower compartment of $A$ ($e_{_A}\in \cE_J$ a verge, $\p=\supp(A)$) carries a condition. We say in this case that the corresponding orbit module $\C \cO_A^J$ has {\bf full condition set}. 
\end{Remark}

\section{$\C \cO_A^J$ as constituent of $[A]\C U$}\label{sec7}
Let $\p=\{(i_1,j_1), (i_2, j_2), \ldots, (i_m, j_m)\}\subseteq \Phi^-, 1\leqslant i_1<i_2<\cdots<i_m\leqslant n$ be a completely hook disconnected main condition set in $\Phi^-$, and let $A\in M_n(q)$ be such that $\supp (A)=\p$.
Let $\lambda=(n-m,m)$
 and let $\s$ be the unique row standard $\lambda$-tableau  with $\ss=(i_1, \ldots, i_m)$.
Thus, in view of \ref{6.6}, $e_{_A}\in \cE_J, J=J(\s)$ as in \ref{6.1}, is a verge, $\main(e_{_A})=\p$ and $\C \cO_A^J$ has full condition set. Being an irreducible  $\C U$-module by \ref{results} part (5), $\C \cO_A^J$ must occur as irreducible constituent of $[B]\C U$ for precisely one verge $[B]\in \cE=\cE_{\Phi^-}$, $B\in \Lie(U)=\{u-E\,|\, u\in U\}$. We shall show that $B=A\in V_J\leqslant \Lie(U)$. However, as we shall see, a $\C U$-homomorphism from $\C \cO_A^J$ into $\C \cO_A$,  ($\cO_A=\cO_A^{\Phi^-}$) will not take $e_{_A}\in \cO_A^J$ to $[A]\in \cO_A$.


Throughout this section $A\in V_J, J=J(\s)$, with $\supp(A)=\p$. We use the notation of the previous sections freely. In particular $J=J(\s), L=L(\s)$ and $K=K(\s)$ are as in \ref{6.1}. Recall from \ref{action} that $L$ splits into closed subsets $L_1=\{(i,j)\in \Phi^-\,|\, i,j\notin \ss\}$ and $L_2=\{(i,j)\in \Phi^-\,|\, i,j\in \ss\}$ of $\Phi^-$.

\begin{Defn}\label{fide}
Throughout let $\ff=\hat U_{L_2}$ be the trivial idempotent of $U_{L_2}$, that is $f=q^{-|L_2|}\sum _{x\in U_{L_2}}x$, and set $\hat e_{_A}=e_{_A} \ff \in \C \cO_A^J$.
\end{Defn}

\begin{Lemma}\label{7.2}
 $\hat e_{_A}\neq 0$. Thus $\hat e_{_A}\C U=\C \cO_{_A}^J$.
\end{Lemma}
\begin{proof}
Recall from \ref{results} part (3) \romannumeral2), that $\Stab_{\mbox{\tiny $U_{L_2} $}}(e_{_A})=U_{L_2^0}$. Choose the following linear ordering of $L_2$ and take all products of elements $x_{ij}(\alpha), (i,j)\in L_2$ in this fixed order: First we take the roots $(i,j)\in L_2^0$ in an arbitrary ordering, then the remaining positions $(i,j)$ in $L_2$ along columns top down and rows from left right. Thus $(i,j)<(a,b)$ in $L_2\setminus L_2^0$ implies $i, j, a, b\in \ss$, the main conditions in row $i$ and $a$ are to the left of $(i,j)$ and $(a,b)$ respectively, and $j<b$ or $j=b$ and $i<a$. 

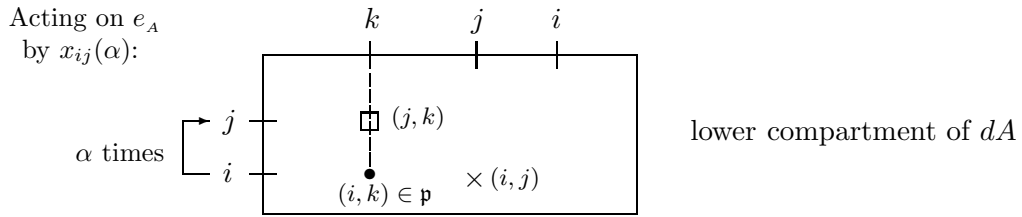
\begin{center}
\begin{picture}(250, 100)
\put(-65, 81){\small Acting on $e_{_A}$}
 \put(-60, 69){\small by $x_{ij}(\alpha)$:}
\put(30,10){\line(1,0){140}}
\put(30,70){\line(1,0){140}}
\put(30,10){\line(0,1){60}}
\put(170,70){\line(0,-1){60}}
\put(70,65){\line(0,1){10}}
\put(68,80) {$k$}
\put(110,65){\line(0,1){10}}
\put(108,80) {$j$}
\put(140,65){\line(0,1){10}}
\put(138,80) {$i$}
\put(25,45){\line(1,0){10}}
\put(15,42) {$j$}
\put(25,25){\line(1,0){10}}
\put(15,22) {$i$}
\put(105,20){$\times$}
\put(115,21){\footnotesize$(i,j)$}
\put(66.5,45){\framebox}
\put(78,44){\footnotesize$(j,k)$}
\put(190,38){lower compartment of $dA$}
\put(10,25){\line(-1,0){10}}
\put(0,25){\line(0,1){20}}
\put(0,45){\vector(1,0){10}}
\put(-40,30){\small$\alpha$ times}

\put(70,25){\circle*{4}}
\put(58,15){\footnotesize$(i,k)\in \p$}
\multiput(70,28)(0,6){6}{\line(0,1){5}}
\end{picture}
\end{center}

$x_{ij}(\alpha)\,$, $(i,j)\in L_2\setminus L_2^0$, acts on $e_{_A}$ by adding $\alpha$ times row $i$ to row $j$ and hence changing only position $(j,k)\in J$ ($k\notin \ss$). In the order above, the root subgroups fill the positions (in $J$) on the top of main conditions  from top down. As a consequence, $e_{_A}\ff$ is $q^{-|L_2|+|L_2^0|}$ times the sum of all lidempotents $e_{_B}$, where $B$ runs through the set of all matrices in $V_J$, coinciding with $A$ except the positions (in $J$)  in columns of and above the main conditions, which are filled by arbitrary entries from $\F_q$. This proves $e_{_A}\ff=q^{-|L_2|+|L_2^0|} \sum e_{_B}\neq 0$ proving the claim.
\end{proof}

Now let $H\leqslant \Pstab_{\mbox{\tiny$U$}}(\hat e_{_A})=\{u\in U\,|\, \hat e_{_A}u=\lambda_u \hat e_{_A}, \,\exists\, \lambda_u \in \C ^*\}$. Then
 $\lambda: H\rightarrow \C^*: u\mapsto \lambda_u$ is a linear character of $H$ and we have a natural epimorphism of $\C U$-modules
\begin{equation}\label{7.3}
\mu:\, \Ind^U_H\C \hat e_{_A}\rightarrow \C \cO_A^J: \hat e_{_A} \otimes u\mapsto \hat e_{_A}u 
\end{equation}
for any $u\in U$. We shall show that $U_{\hat \cR}\leqslant \Pstab_{\mbox{\tiny$U$}}(\hat e_{_A})$, where  $\hat \cR\subseteq \Phi^-$ is defined in \ref{pstab} for orbit $\cO_A=\cO_A^{\Phi^-}\subseteq \cE_{\Phi^-}$ containing the verge lidempotent $[A]$. Then we prove that $U_{\hat \cR}$ acts on $\hat e_{_A}$ by the linear character afforded by $\epsilon_{_A}$ defined in \ref{5.9}. Then  \ref{5.9} says in particular that $\Ind^{\mbox{\footnotesize$U$}}_{\mbox{\footnotesize$U_{\hat \cR}$}}\C \hat e_{_A}$ is irreducible, proving that $\mu$ in (\ref{7.3}) with $H=U_{\hat \cR}$ must be an isomorphism. This implies $\Pstab_{\mbox{\tiny$U$}}(\hat e_{_A})=U_{\hat \cR}$ and  identifies the irreducible $\C U$-module $\C \cO_A^J$ as the  unique irreducible constituent in $\C \cO_A^{\Phi^-}=[A]\C U$ corresponding to the trivial module of $\C (U_{\hat \cR^-}/U_{\cR^o})$ extended to  $\End_{\C U}(\C \cO_A^{\Phi^-})$ by $\theta_{_A}$ as in \ref{5.8} part 3).

\begin{Lemma}\label{7.4}
$U_{L_2}\leqslant \Stab_{\mbox{\tiny$U$}}(\hat e_{_A})$ and $L_2\in \cR^o$. Moreover $U_{L_2}\leqslant \Stab_{\mbox{\tiny$U$}}[A]$ as well.
\end{Lemma}
\begin{proof}
By construction $\hat e_{_A}u=e_{_A}\ff u=e_{_A}\ff=\hat e_{_A}$ for all $u\in U_{L_2}$. 
So $U_{L_2}\leqslant \Stab_{\mbox{\tiny$U$}}(\hat e_{_A})$. For $(i,j)\in L_2$, we have $i,j\in \ss$ and hence $j\notin \pj\subseteq J$. Thus column $j$ is a zero column in $A$ and hence $(i,j)\in \cR^o$ by \ref{pstab}.
\end{proof}

Let $(a,b)\in L_2$, $(i,j)\in L_1$. Then $a,b\in \ss, i,j\notin \ss$ and hence $a\neq j$ and $b\neq i$. By Chevalley's commutator formula (see e.g. [\cite{carter}])
$X_{ab}$ and $X_{ij}$ commute and hence $U_L=U_{L_1}\times U_{L_2}$. In particular this implies
\begin{Lemma}\label{7.5}
$U_{L_1^0} =\Stab_{\mbox{\tiny$U_{L_1}$}}( e_{_A})\leqslant  \Stab_{\mbox{\tiny$U_{L_1}$}}( \hat e_{_A}) $. Moreover $L_1^0\subseteq \cR^0$ and hence $U_{L_1^0}\leqslant  \Stab_{\mbox{\tiny$U$}}[A]$.
\end{Lemma}
\begin{proof}
This follows from \ref{results} part (3) \romannumeral1).
\end{proof}

\begin{Defn}\label{7.6}
Let $L_1^1$ be the set of all positions $(i,j)\in L_1$ such that $e_{_A}x_{ij}(\alpha)\, (\alpha\in \F_q) $ change entries at main hook intersections. Thus $(i,j)\in L_1^1$ if and only if $i,j\in \pj$ and there exist $s, t\in \ppi$  such that $n\geqslant s>t>i>j\geqslant 1$ and $(s,i), (t,j)\in \p$.
\end{Defn}

\begin{Lemma}\label{7.7}
Let $(i,j)\in L_1^1$. Then $X_{ij}\leqslant \Stab_{\mbox{\tiny$U$}}( \hat e_{_A}) $. Moreover $(i,j)\in \hat \cR\setminus \cR=\hat \cR^-\setminus \cR^o$.
\end{Lemma}
\begin{proof}
By definition of $L_1^1$ and \ref{results} part  (3) \romannumeral3) we find $(s,t)\in L_2$ and $\beta\in \F_q^*$ such that $e_{_A}x_{st}(\beta)=e_{_A}x_{ij}(\alpha)$ for   $\alpha\in \F_q^*$. Thus
$$
\hat e_{_A} x_{ij}(\alpha)=e_{_A}\ff x_{ij}(\alpha)=e_{_A} x_{ij}(\alpha)\ff=e_{_A} x_{st}(\beta)\ff=e_{_A}\ff=\hat e_{_A},
$$
since $x_{st}(\beta)\ff=\ff x_{st}(\beta)=f$. Thus $X_{ij}\leqslant \Stab_{\mbox{\tiny$U$}}( \hat e_{_A}) $. Now $[A]x_{ij}(\alpha)=\theta(\alpha A_{ij})[B]=[B]\in \cO_A$, since $A_{ij}=0$, where $B$ is obtained from $A$ by adding $-\alpha $ times column $j$ to column $i$ in $A$ and projecting the resulting matrix into $\Lie(U)$. Obviously this is the same matrix $B$ occurring in $e_{_A}x_{ij}(\alpha)=e_{_B}\in \cO_A^J$, differing from $A$ only on the main hook intersection $(t,i)\in J$. By \ref{pstab}, $(i,j)\in \hat \cR\setminus \cR=\hat \cR^-\setminus \cR^o$, as desired.
\end{proof}

Now let $J^0\subseteq J$ be the set of all positions $(a,j)$ in the lower compartment of $dA$ in column $j$ which are in zero columns of $A$ if  $j\notin \ss$ or,  there is a main condition $(b,j)$ above or on it. Thus  $J^0=\{(a,j)\in J\,|\, j\notin \pj \text{ or } \exists\, (b,j)\in \p \text{ with } b\leqslant a\}$.

\begin{Lemma}\label{7.8}
Let $(a,j)\in J^0$. Then $X_{aj}\leqslant \Pstab_{\mbox{\tiny$U$}}( \hat e_{_A})$ and $X_{aj}\leqslant\Stab_{\mbox{\tiny$U$}}( \hat e_{_A})$ if and only if $(a,j)\notin \p$. Moreover $(a,j)\in \cR$ and $(a,j)\in \cR^o$ if and only if $(a,j)\notin \p$. Finally $X_{aj}$ acts on $\hat e_{_A}$ by the same linear character as on $[A]\in \cE$.
\end{Lemma}
\begin{proof}
Recall from the proof of \ref{7.2} that $\hat e_{_A}=e_{_A}\ff=q^{-|L_2|+|L_2^0|}\sum_{B} e_{_B}$, where $B$ runs through the set of all matrices in $V_J$ coinciding with $A$ at positions in $\p$, having zero entries at all positions in zero columns of $A$ and in columns of $A$ with main conditions below those. Let $(a,j)\in J, \alpha\in \F_q$, then by \ref{action} and \ref{2.17}
\[
e_{_B}x_{aj}(\alpha)=\theta(\alpha B_{aj})e_{_B}=
\begin{cases}
e_{_B} & \text{ for } (a,j)\notin \p\\
\theta(\alpha A_{aj}) e_{_B}& \text{ for } (a,j)\in \p
\end{cases}
\]
for such matrices $B$, since then
$B_{aj}\neq0$ only if $(a,j)\in \p$ and then $B_{aj}=A_{aj}$. Thus
\begin{equation}\label{7.9}
\hat e_{_A}x_{aj}(\alpha)=
\begin{cases}
\hat e_{_A} & \text{ for } (a,j)\notin \p\\
\theta(\alpha A_{aj}) \hat e_{_A}& \text{ for } (a,j)\in \p.
\end{cases}
\end{equation}
By \ref{pstab} again $X_{aj}\in \cR^o$, if $(a,j)\notin \p$ and if $(a,j)\in \p$ then $[A]x_{aj}(\alpha)=\theta(\alpha A_{aj})[A]$.
\end{proof}

Let $I=I(\s)=\{(i,j)\in \Phi^-\,|\,i\notin \ss, j\in \ss\}$ (compare \ref{4.4} part 3)). Then $I$
 is  closed in $\Phi^-$ and $\Phi^-$ is the disjoint union of $K$ and $I$. So $U=U_KU_I=U_IU_K$ (but in general neither $U_K$ nor $U_I$ is normal in $U$). In general $U_I$ does not act monomially on the lidempotent basis $\cO_A^J$ of $\C \cO_A^J$, but, as we shall show, $U_I$ is contained in $\Stab_{\mbox{\tiny$U$}}(\hat e_{_A})$. To prove this recall $M_\s\cong \C \cE_J$ has $\C$-basis $\{\overline{P_\lambda} du\,|\, u\in U_J\}, d=d(\s)\in \cD_\lambda, U_J=(U_\lambda^-)^d\cap U$, by \ref{mackey} part 2) and \ref{4.5} part 1). Clearly the image of $\hat e_{_A}=e_{_A}\ff$ in $M_\s$ under the isomorphism $f^*: \C \cE_J\rightarrow M_\s$ in \ref{results} part (5) is contained in $M_\s \ff$ which is generated as $\C$-vector space by $\{\overline{P_\lambda} du\ff\,|\, u\in U_J\}$, where $\overline{P_\lambda}=\sum_{x\in P_\lambda}x$. We show for $g\in U_I,$ that $\overline{P_\lambda} du\ff g=\overline{P_\lambda} du\ff$, proving $m\ff g=m \ff,\,\forall\, m\in M_\s$. From this follows in particular $\hat e_{_A}g=\hat e_{_A}$, that is $U_I\leqslant \Stab_{\mbox{\tiny$U$}}(\hat e_{_A})$.


We first inspect matrices of the form $duv$ for $u\in U_J, v\in U_ {L_2}$. Recall that we may think of $du$ as $u$ with reordered rows, dividing  $u$ into two compartments, the rows of the lower one labelled by $i_1<\cdots<  i_m$, those of the upper one by the numbers $1\leqslant j \leqslant n$ not contained in $\ppi=\{i_1,\ldots, i_m \}$ in their natural order. Note that with these convention, the ``last ones'' of \ref{xlambda}, coming from the diagonal ones in $u$, are at position $(i,i)$ in $du$ for $1\leqslant i \leqslant n$.


Let $u\in U_J, (a,b)\in L_2$ and $\beta\in \F_q$. Then $dux_{ab}(\beta)$ is obtained from $du$ by adding $\beta$ times column $a$ to column $b$. Note that in $du$ the only non-zero entries in columns $b$ and $a$ are the ``last ones'' at positions $(b,b)$ and $(a,a)$, since $u\in U_J, J=\{(r,s)\in \Phi^-\,|\,r\in \ss, \s\notin \ss\}$. So the matrix $dux_{ab}(\beta)$ coincides with $du$ in all positions but positions $(a,b)$, on which the entry of $dux_{ab}(\beta)$ is $\beta$:  
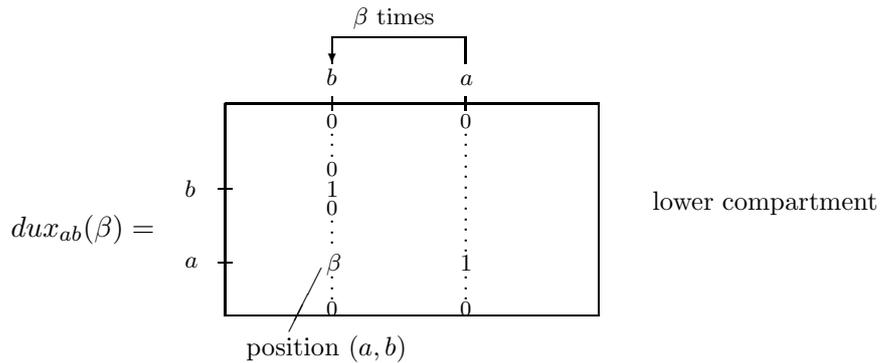
\begin{center}
\begin{picture}(290,110)
\put(0,40){$dux_{ab}(\beta)=$}

\put(80,10){\line(1,0){140}}
\put(80,90){\line(1,0){140}}
\put(80,10){\line(0,1){80}}
\put(220,10){\line(0,1){80}}

\put(77.5,30){\line(1,0){5}}

\put(120,87.5){\line(0,1){5}}

\put(170,87.5){\line(0,1){5}}

\put(118,97){\footnotesize$b$}
\put(168,97){\footnotesize$a$}
\put(65,28){\footnotesize$a$}
\put(65,55){\footnotesize$b$}

\put(170,105){\line(0,1){10}}
\put(170,115){\line(-1,0){50}}
\put(120,115){\vector(0,-1){10}}
\put(128,120){\footnotesize$\beta$ times}

\put(118,81){\scriptsize$0$}
\put(119,70){\scriptsize$\vdots$}
\put(118,63){\scriptsize0}
\put(118,55){\footnotesize1}
\put(77.5,58){\line(1,0){5}}
\put(118,48){\scriptsize0}
\put(119,37){\scriptsize$\vdots$}
\put(118,27){\footnotesize$\beta$}
\put(119,16){\scriptsize$\vdots$}
\put(118,10){\scriptsize0}

\put(168,27){\footnotesize1}
\put(169,16){\scriptsize$\vdots$}
\put(168,10){\scriptsize0}

\put(168,81){\scriptsize$0$}
\put(169,70){\scriptsize$\vdots$}
\put(169,58){\scriptsize$\vdots$}
\put(169,46){\scriptsize$\vdots$}
\put(169,34){\scriptsize$\vdots$}

\put(240,50){\small lower compartment}

\put(116,28){\line(-1,-2){12}}
\put(88,-5){\small position $(a,b)$}
\end{picture}
\end{center}

Let $\fD_{L_2}$ be the set of all matrices in $M_n(q)$ by placing arbitrary values from $\F_q$ in $du$ at positions $(a,b)\in L_2$. We have shown:

\begin{Lemma}\label{7.10}
Keeping the notation introduced above we have 
\[
du\ff=q^{-|L_2|}\sum _{v\in U_{L_2}} duv=q^{-|L_2|} \sum_{y\in \fD_{L_2}} y. 
\]
\hfill $\square$
\end{Lemma}

Fix $u\in U_J$ and define $\fD_{L_2}\subseteq M_n(q)$ as in  \ref{7.10}. Let $y\in \fD_{L_2}, (i,j)\in I$ and $\alpha\in \F_q$. Then we have:
\begin{Lemma}\label{7.11}
There exists $z\in P_\lambda$ (depending on $x_{ij}(\alpha)$) such that $zyx_{ij}(\alpha)=\tilde y\in \fD_{L_2}$.
\end{Lemma}
\begin{proof}
Since $(i,j)\in I$ we have $i\notin \ss, j\in \ss$. Note that $yx_{ij}(\alpha)$ is obtained from $y$ by adding $\alpha$ times column $i$ to column $j$ in $y$. Note further, that the only non-zero entries of $y$ besides the last ones are all in the lower compartment and  on positions $(s,t)$ with $s>t$ and  $s\in \ppi$, $i.e.$ the ``last one'' at position $(s,s)$ belongs to the lower compartment. So:
\begin{equation}
\begin{picture}(310,230)
\put(0,100){$yx_{ij}(\alpha)=$}

\put(80,20){\line(0,1){160}}
\put(240,20){\line(0,1){160}}
\put(80,20){\line(1,0){160}}
\put(240,180){\line(-1,0){160}}

\put(80,85){\line(1,0){160}}

\put(77.5,60){\line(1,0){5}}
\put(70,58){\small$j$}

\put(77.5,140){\line(1,0){5}}
\put(70,136){\small$i$}

\put(150,177.5){\line(0,1){5}}
\put(148,190){\small$j$}
\put(210,177.5){\line(0,1){5}}
\put(208,190){\small$i$}

\put(210,200){\line(0,1){10}}
\put(210,210){\line(-1,0){60}}
\put(150,210){\vector(0,-1){10}}

\put(162,215){\small$\alpha$ times}

\put(147.5,137){\small$\alpha$}
\put(207.5,135){\small$1$}
\put(207.5,125){\small$0$}

\put(147.5,57){\small$1$}
\put(147.5,47){\small$*$}
\multiput(150,45)(0,-4){5}{\line(0,-1){2}}
\put(147.5,21){\small$*$}

\put(207.5,32){\small$*$}
\multiput(210,31)(0,-2){2}{\line(0,-1){1}}
\put(207.5,21){\small$*$}
\put(207.5,38){\small$0$}

\multiput(210,122)(0,-4){19}{\line(0,-1){2}}

\put(260,130){\small upper compartment}

\put(260,50){\small lower compartment}

\put(214,30){\line(1,-1){18}}
\put(228,5){\small positions in $J$}
\end{picture}
\end{equation}

Thus $yx_{ij}(\alpha)$ has $\alpha$ at position $(i,j)$. Now since $j<i, X_{ji}\subseteq U^+\leqslant P_\lambda$, hence $x_{ji}(-\alpha) y x_{ij}(\alpha)$ is obtained from $yx_{ij}(\alpha)$ by adding $-\alpha$ times row $j$ to row $i$, removing entry $\alpha$ at position $(i,j)$ again. This might introduce non-zero entries in row $i$ to the left of position $(i,j)$ But  those  can be removed by row operations coming from multiplication from the left by element $x_{si}(\gamma), \gamma\in \F_q$ where either $s\notin \ss$ (so $X_{si}\leqslant L_\lambda$) or $X_{si}(\sigma)\leqslant U_\lambda^+$, the unipotent radical of $P_\lambda$, if $s\in \ss, s<j$. Note that   besides the last one at position $(i,i)$, column $i$ of $y$ has non-zero entries only at positions $(b,i)$ with $i<b\in \ppi,$ that is to the left of the last one at position $(b,b)$, therefore the resulting matrix $\tilde y=zyx_{ij}(\alpha), z\in P_\lambda$, differs from $y$ only at positions in column $j$ below position $(j,j)$, which all belong to $L_2$. Then $\tilde y\in \fD_{L_2}$ again, as desired.
\end{proof}

\begin{Cor}\label{7.13}
Let $u\in U_J$, and let $\fD_{L_2}$ be defined for $u$ as above. Then $\overline{P_\lambda} du \ff g =\overline{P_\lambda} d u \ff$ for all $g\in U_I$.
\end{Cor}
\begin{proof}
Since $U_I$ is generated by $X_{ij}$, $(i,j)\in I$, we may assume $g=x_{ij}(\alpha)$ for some $(i,j)\in I$ and $\alpha\in \F_q$. By \ref{7.11} we have for each $y\in \fD_{L_2}$,\,
$\overline{P_\lambda} y x_{ij}(\alpha)=\overline{P_\lambda}z^{-1}\tilde y=\overline{P_\lambda}\tilde y$. Clearly $y\rightarrow \tilde y$ is a permutation of $\fD_{L_2}$ and hence by \ref{7.10}
\[
\overline{P_\lambda} du \ff x_{ij}(\alpha)=q^{-|L_2|} \sum_{y\in \fD_{L_2}} \overline{P_\lambda}y x_{ij}(\alpha)=q^{-|L_2|} \sum_{y\in \fD_{L_2}}\overline{P_\lambda}\tilde y=\overline{P_\lambda} du \ff.
\]
\end{proof}

Now our desired result follows immediately, observing that $f^*(\hat e_{_A})\in M_\s\ff$:
\begin{Cor}\label{7.14}
Let $g\in U_I$. Then $\hat e_{_A}g=\hat e_{_A}$ and hence $g\in \Stab_{\mbox{\tiny$U$}}( \hat e_{_A})$. \hfill $\square$
\end{Cor}

Let $(i,j)\in I$, then $i\notin \ss$ and $j\in \ss$. In particular, since $\p\subseteq J=\{(a,b)\in \Phi^-\,|\, a\in \ss, b\notin \ss\}$, $(r,j)\notin \p$ for all $j<r\leqslant n$ and hence column $j$ is a zero column in $A$. By \ref{pstab} $I\subseteq \cR^o\subseteq \hat \cR$. Indeed we have the following:

\begin{Lemma}\label{7.15}
$\hat \cR=L_2\cup L_1^0 \cup L_1^1 \cup J^0 \cup I$ with $\hat \cR\setminus \cR=L_1^1,\, \cR=L_2\cup L_1^0 \cup J^0 \cup I$ and $\cR^0=L_2\cup L_1^0   \cup I\cup (J^0\setminus\p)$.
\end{Lemma}
\begin{proof}
We have already seen in the previous results that the right hand sides are contained in the left hand sides of all equalities in the lemma. For $(s,i), (t,j)\in \p$ with $s>t>i>j$ there is a
 main hook intersection $(t,i)\in J$, since $t\in \ss$ and $i\notin \ss$. From this follows immediately that $\hat \cR\setminus \cR=L_1^1$. Moreover since $\p\subseteq J^0$, it suffices to check $\cR^o\subseteq L_2\cup L_1^0   \cup I\cup (J^0\setminus\p)$.


Let $(i,j)\in \cR^0$. If $j\in \ss$, then $(i,j)\in I$ if $i\notin \ss$. So let $i\in \ss$. Then $(i,j)\in L_2$.
Now suppose $j\notin \ss$. If $i\in \ss$, then $(i,j)\in J$. If $j\notin \pj$, then $j$ is a zero column in $A$ and $(i,j)\in J^0$. If $j\in \pj$ then there exists a main condition $(b,j)\in \p$  with $b<a$ by the definition of $\cR^0$ and hence $(i,j)\in J^0$ by the definition of $J^0$. Finally let $i,j\notin \ss$. Then $(i,j)\in L_1$. If $j\notin \pj$, then column $j$ is a zero column in $A$ and $(i,j)\in L_1^0$ by \ref{results} part (3) \romannumeral1). If $j\in \pj$, there is a main condition $(a,j)\in \p$ in column $j$ above $(i,j),$ that is $a<i$ by the definition of $\cR^o$ and hence $(i,j)\in L_1^0$ again by  \ref{results} part (3) \romannumeral1).
\end{proof}

Recall that $\p=\{(i_1, j_1), \ldots, (i_m,j_m)\}$ is completely hook disconnected. Thus we may apply the results in \ref{5.8}  to the $U$-orbit module $\C \cO_A=[A]\C U$. Recall from \ref{5.9} that there exists a linear character $\psi_{_A}$ whose restriction to $U_ {\hat \cR^ -}\trianglelefteq U_{\hat \cR}$ is trivial and which is $\theta_A$ defined in \ref{pstab} on $U_{\hat \cR}/U_{\hat \cR^-}\cong X_{i_1 j_1}\times \cdots \times X_{i_m j_m}$. Note that by \ref{7.8} and \ref{7.15} this is precisely the linear character of $U_{\hat \cR}$ afforded by $\C \hat e_{_A}$. Let $\epsilon_{_A}\in \C U_{\hat \cR}$ the primitive idempotent such that $\C \epsilon_{_A}$ affords $\psi_{_A}$ and set $\widehat{[A]}=[A]\epsilon_{_A}\in \C \cO_A$. Then $S=\widehat{[A]}\C U\leqslant \C \cO_A$ is the induced $\C U$-module $\Ind^U_{U_{\hat \cR}}\C \widehat{[A]}$ and is irreducible. Since $\C \hat e_{_A}\cong\C \widehat{[A]}$ as $\C U_{\hat \cR}$-modules, we conclude that $\Ind^U_{U_{\hat \cR}}\C \hat e_{_A}$ is irreducible too and hence the map $\mu$ in \ref{7.3}
\begin{equation}
\mu:\, \Ind^U_{U_{\hat \cR}}\C \hat e_{_A}\rightarrow \hat e_{_A}\C U= \C \cO_A^J
\end{equation}
is an $\C U$-module isomorphism. Thus we have shown:

\begin{Theorem}\label{7.16}
Let $\lambda=(n-m,m)$ be a composition of $n$, $\s\in \RStd(\lambda), J=J(\s)$ and let $e_{_A}\in \cE_J$ be a verge with $\p=\supp(A)$. Then $\C \cO_A^J=e_{_A}\C U$ is an irreducible $\C U$-module  isomorphic to the unique irreducible constituent of $[A]\C U=\C \cO_A$ corresponding to the trivial representation of $U_{\hat \cR^-}$ extended to $U_{\hat \cR}$ by the linear character $\theta_{_A}$ of $\mbox{\huge$\times$}\!_{(i,j)\in \p} X_{ij}$. Conversely, if $\p\subseteq \Phi^-$ is a set of completely hook disconnected main conditions, $A\in M_n(q)$ with $\supp(A)=\p$, then the unique irreducible constituent of $\C \cO_A$ described above is isomorphic to $\C\cO_A^J$, where $\lambda=(n-|\p|, |\p|)$. \hfill $\square$
\end{Theorem}

There are several sequences of \ref{7.16}:
\begin{Cons}\label{7.17}
\begin{enumerate}
\item [1)]   Let $\lambda, \mu$ be 2-part partitions of $n$ and let $\s\in \RStd(\lambda), \t\in \RStd(\mu)$. Let $e_{_A}\in \cE_{J(\s)}$ and $e_{_B}\in \cE_{J(\t)}$ be verge lidempotents. Then  $\C \cO_A^{J(\s)}\cong \C \cO_B^{J(\t)}$ if and only if $A=
B$.
\item [2)] Let $\lambda=(n-m,m)\vDash n$ and let $\s\in \RStd(\lambda)$. Then $M_\s$ is  multiplicity free, its irreducible constituents being of the form $\C \cO_A^J,\, J=J(\s)$, where $A\in V_J$ satisfies: $\p=\supp(A)$ is a main condition set with $\ppi\subseteq \ss, \pj\cap \ss=\emptyset$. By [\cite{guo}, 2.5.10] $\dim _\C (\C \cO_A^J)=|\cO_A^J|$ depends only on $\p$, not on $A\in V_J$ with $\supp(A)=\p$, and hence there are $(q-1)^{|\p|}$ many orbits $\cO_A^J\subseteq \cE_J$ with $\supp(A)=\p$ for a given completely hook disconnected main condition set fitting $\s\in \RStd(\lambda)$. As consequence, the number of irreducible constituents of $M_\s$ of fixed dimension $q^c$ is a polynomial in $(q-1)$ with integral, non-negative coefficients independent of $q$.
\item [3)] Each permutation representation of $GL_n(q)$ on the cosets of a maximal parabolic subgroup is isomorphic to some $M(\lambda)$ of some partition $\lambda=(n-m,m)$ of $n$. Let $\fM_\lambda$ be the set of completely hook disconnected main condition sets $\p$, which fit at least one row standard $\lambda$-tableaux and for $\p\in \fM_\lambda$, let $k_{\lambda, \p}$ be the number of distinct $\s\in \RStd(\lambda)$ such that $\p$ fits $\s$. Then $k_{\lambda, \p}$ is independent of $q$  and 
\[
\Res^{GL_n(q)}_U M(\lambda)=\bigoplus_{\p\in \fM_\lambda} \bigoplus_{\stackrel{A\in M_n(q)}{\supp(A)=\p}} ([A]\epsilon_{_A} \C U)^{k_{\lambda, \p}}.
\] 
\end{enumerate}
\end{Cons}


\providecommand{\bysame}{\leavevmode ---\ }
\providecommand{\og}{``} \providecommand{\fg}{''}
\providecommand{\smfandname}{and}
\providecommand{\smfedsname}{\'eds.}
\providecommand{\smfedname}{\'ed.}
\providecommand{\smfmastersthesisname}{M\'emoire}
\providecommand{\smfphdthesisname}{Th\`ese}


\begin{thebibliography}{LYLE}
\addcontentsline{toc}{chapter}{Bibliography}

\bibitem{andre1}
{\scshape C.A.M. Andr\'{e}}, {\og Basic characters of the
unitriangular group\fg},  \emph{J. Algebra},  \textbf{175},
287-319,(1995) .

 \bibitem{carter}
 {\scshape R. W. Carter}, {\og Simple groups of Lie type\fg}, John Wiley \& Sons, \emph{London, New York, Sydney, Toronto}, (1972).

\bibitem{brandt}
{\scshape M. Brandt}, {\og   On unipotent Specht modules of general linear groups\fg}, \emph{PhD Thesis, Universit\"{a}t Stuttgart},  (2004).

\bibitem{brandt2}
{\scshape M. Brandt, R. Dipper, G. James and S. Lyle}, {\og   Rank polynomials\fg}, \emph{Proc. London Math. Soc.}, \textbf{ (3). 98}, {1-18},  (2009).



\bibitem{super}
{\scshape P. Diaconis and I. M. Isaacs}, {\og Supercharacters and superclasses for algebra groups\fg}, \emph{Trans. Amer. Math. Soc.},  \textbf{360(5)}, 2359--2392, (2008).

\bibitem{mindim}
{\scshape R. Dipper and Q. Guo}, {\og Irreducible constituents of minimal degree in supercharacters of the finite unitriangular groups\fg}, \emph{Journal of Pure and Applied Algebra}, \textbf{219 (7)}, 2559--2580, (2015).


\bibitem{dj1}
{\scshape R. Dipper and G. James}, {\og On Specht modules for general linear groups\fg}, \emph{J. Algebra},  \textbf{275}, 106--142, (2004).

\bibitem{Gr55}
{\scshape J.A. Green}, {\og The characters of the finite general linear groups\fg}, \emph{Trans. Amer. Math. Soc.\fg}, \textbf{80(2)}, 402--447, (1955).


\bibitem{guo}
{\scshape Q. Guo}, {\og On the $U$-module structure of the unipotent Specht modules of finite general linear groups}, to appear in \emph{J. Algebra}, DOI: 10.1016/j.jalgebra.2016.02.015.



\bibitem{james}
{\scshape G. James}, {\og The representation theory of the symmetric groups\fg}, \emph{Lecture Notes in Math.},  \textbf{308}, Springer-Verlag, Berlin and New York, (1973).

\bibitem{jamesbook}
{\scshape G. James}, {``Representations of general linear groups''},  \emph{LMS Lecture Notes}, 94, (1984).



\bibitem{markus}
{\scshape M. Jedlitschky}, {\og Decomposing Andr\'{e}-Neto supercharacters for
Sylow $p$-subgroups of type $D$\fg}, \emph{PhD. Thesis, Universit\"{a}t Stuttgart}, (2013).




\bibitem{kirillov}
{\scshape A. A. Kirillov}, {\og Lectures on the orbit method \fg},  \emph{AMS},  \textbf{64},
(2004).


\bibitem{scott}
{\scshape S. Andrews}, {\og The irreducible unipotent modules of the finite general linear groups via tableaux\fg}, arXiv:1502.06542, (2015).

\bibitem{yan}
{\scshape N. Yan}, {\og Representations of finite unipotent linear groups by the method of Clusters\fg}, arXiv:1004.2674v1, (2010).

\end{thebibliography}
\end{document}